\documentclass[11pt]{article}
\usepackage{amsmath}
\usepackage{amsfonts}
\usepackage{amsthm}
\usepackage{amssymb, setspace}
\usepackage{color,graphicx,epsfig,geometry,hyperref,fancyhdr}
\usepackage[T1]{fontenc}
\usepackage{float}
\usepackage{subfig}
\usepackage{commath}
\usepackage{bbm}
\usepackage{mathrsfs,fleqn}
\newtheorem{theorem}{Theorem}[section]

\newtheorem{lemma}{Lemma}[section]

\newcommand{\N}{\mathbb{N}}

\newcommand{\R}{\mathbb{R}}
\newcommand{\C}{\mathbb{C}}

\newcommand{\brac}[1]{\left\{#1\right\}}

\usepackage{mathtools}% For `multlined' and loads `amsmath'
\newcommand*\phantomrel[1]{\mathrel{\phantom{#1}}}% My preferred typesetting

\begin{document}
%\lhead{}
%\rhead{}

\begin{flushleft}
\Large 
\noindent{\bf \Large Reciprocity gap functional methods for potentials/sources with small volume support for two elliptic equations}
\end{flushleft}

\vspace{0.2in}

{\bf  \large Govanni Granados and Isaac Harris}\\
\indent {\small Department of Mathematics, Purdue University, West Lafayette, IN 47907 }\\
\indent {\small Email:  \texttt{ggranad@purdue.edu}  and \texttt{harri814@purdue.edu} }\\

%\vspace{0.2in}

%%%%%%%%%%%%%%%%%%%%%%%%%%%%%%%%%%%%%%%%
\begin{abstract}
\noindent In this paper, we consider inverse shape problems coming from diffuse optical tomography and inverse scattering. In both problems, our goal is to reconstruct small volume interior regions from measured data on the exterior surface of an object. In order to achieve this, we will derive an asymptotic expansion of the reciprocity gap functional associated with each problem. The reciprocity gap functional takes in the measured Cauchy data on the exterior surface of the object. In diffuse optical tomography, we prove that a MUSIC-type algorithm can be used to recover the unknown subregions. This gives an analytically rigorous and computationally simple method for recovering the small volume regions. For the problem coming from inverse scattering, we recover the subregions of interest via a direct sampling method. The direct sampling method presented here allows use to accurately recover the small volume region from one pair of Cauchy data. We also prove that the direct sampling method is stable with respect to noisy data. Numerical examples will be presented for both cases in two dimensions where the measurement surface is the unit circle.
\end{abstract}

\noindent {\bf Keywords}: Diffuse Optical Tomography $\cdot$ Inverse Scattering $\cdot$ MUSIC Algorithm $\cdot$ Direct Sampling \\

\noindent {\bf MSC}: 35J05, 35J25

%%%%%%%%%%%%%%%%%%%%%%%%%%%%%%%%%%%%%%%%%%%%%%%%%%%%%%%%%%%
\section{Introduction}
The two problems we consider in this paper are motivated by diffuse optical tomography (DOT) and inverse scattering theory. In both problems, the goal is to reconstruct interior subregions of small volume from known Cauchy data on the boundary of the given bounded open set $\Omega$ in $\mathbb{R}^2$ or $\mathbb{R}^3$. These are inverse shape problems where the knowledge of the solution to a partial differential equation on the boundary is used to recover unknown interior regions. Here we are interested in reconstructing a subregion $D \subset \Omega$ such that dist$(D , \partial \Omega)>0$. In our models, a Dirichlet condition is imposed on the exterior boundary $\partial \Omega$ and the corresponding Neumann condition is measured. For the entirety of this paper, we assume that $D$ is a collection of small volume subregions such that $|D| = \mathcal{O}(\epsilon^d)$, where $d$ = 2 or 3 is the dimension. To fix the notation, we let 
\begin{equation}\label{sball}
D = \bigcup_{j=1}^{J} D_j \quad \text{with} \quad D_j = (x_j + \epsilon B_j) \quad \text{such that} \quad \text{dist}(x_i , x_j) \geq c_0>0
\end{equation}
for $ i \neq j$ where the parameter $0< \epsilon \ll 1$ and $B_j$ is a domain with Lipschitz boundary centered at the origin such that $|B_j| = \mathcal{O}(1)$. We also assume that the individual regions $D_j$ are disjoint. See Figure \ref{eg-geo} for a visual representation of the described set up. 
\begin{figure}[ht]
\centering
\includegraphics[scale=0.25]{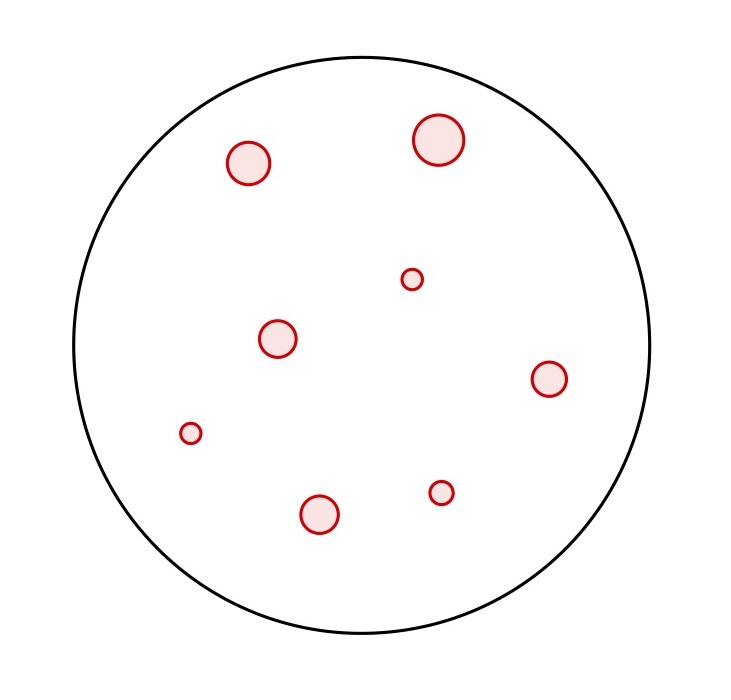}
\caption{Here is an example of the type of small volume regions $D$ described in \eqref{sball} contained in a circle $\Omega$ that will be considered throughout the paper.}
\end{figure}\label{eg-geo}

In DOT, the propagation of light through a medium is modeled by the steady-state diffusion equation. Inside the medium, we consider the case where the absorption coefficient is zero except in the small volume subregions. In this case, the Cauchy data corresponds to inward and outward light fluxes across the medium's surface. For a comprehensive description of DOT see topical reviews \cite{DOT-descrip1, DOT-descrip2}. In our inverse scattering problem, a forcing term will be applied to the direct scattering problem where the source term is zero except in the small volume subregions. Here, the Dirichlet condition represents the scattered field on the surface of the exterior boundary $\partial \Omega$. See \cite{Kress-Rundell,Ren-Zhong,Zhang-Guo,Zhang-Guo2} for more discussion on the theory and applications of this inverse scattering problem.

In order to solve both inverse shape problems, we will develop reconstruction algorithms that fall under the category of {\it qualitative methods}. In many applications, qualitative methods are optimal since one of their main advantages is that they generally require little a priori knowledge of the unknown region $D$. Whereas iterative methods usually require a priori information to construct a ``good'' initial estimate for the unknown region and/or parameters to ensure that the iterative process will converge to the unique solution of the inverse problem. Iterative methods can also be computationally expensive as well as highly ill-conditioned. To avoid requiring a priori knowledge of the small volume regions, we will analyze two qualitative methods. From the given Dirichlet data, we will assume that we have the corresponding normal derivative on the surface $\partial \Omega$ and analyze its asymptotic expansion with respect to the small parameter $\epsilon>0$. 

In our DOT problem, we reconstruct $D$ by developing a MUSIC-type algorithm. This method has been used in many imaging modalities such as acoustic \cite{MUSIC-ammari-scattering,MUSIC-jake,MUSIC-kirsch,MUSIC-park}, electromagnetic \cite{MUSIC-EM1,MUSIC-EM2,MUSIC-parkEM}, and elastic \cite{MUSIC-elastic,MUSIC-elastic2} inverse scattering. Recently, the factorization method was applied to this problem for recovering extended regions in \cite{harris1}. For our second problem associated with inverse scattering, we derive a direct sampling method which is similar to the orthogonality sampling method and reverse time migration, to reconstruct $D$. This method has been widely studied for far-field measurements for several inverse scattering problems, see for e.g. \cite{DSM-harris-dlnguyen,DSM-liu1,dlnguyen1}. These methods have also been applied to problems in DOT \cite{DSM-DOT} and Electrical Impedance Tomography \cite{DSM-EIT}. The aforementioned approaches will be in combination with the so called reciprocity gap functional defined as the surface integral\begin{equation}\label{rgf}
R[v] := \int_{\partial \Omega} v \partial_\nu u - u \partial_\nu v \, \text{d}s.
\end{equation}
In the two problems we are considering, the solution $u \in H^{1}(\Omega)$ with $L^2$ Laplacian represents the respective fields and $v \in H^{1}(\Omega)$ with $L^2$ Laplacian represents a solution to the problem without the small volume regions. This functional has been studied in \cite{rgf-colton} for another inverse scattering problem. We utilize this functional in our asymptotic analysis in order to reconstruct the unknown region with little a priori knowledge.

The rest of the paper is organized as follows. In Section \ref{DOT-section}, we consider an inverse shape problem in DOT and develop the analytic framework for the MUSIC-type algorithm, which requires multiple measurements. To do this, we apply the reciprocity gap functional to a harmonic lifting of the Dirichlet data and the measured Cauchy data in order to derive an imaging functional. We proceed in Section \ref{InvScat-section} by considering the problem from inverse scattering where we derive and analyze a direct sampling imaging functional. This method only requires a single Cauchy pair and we show its stability with respect to error. Here the reciprocity gap functional is applied to a plane wave and the measured Cauchy data. In Sections \ref{DOT-section} and \ref{InvScat-section} numerical examples are presented in $\mathbb{R}^2$ to validate the analysis of the constructed imaging functionals which is based on the asymptotic expansion of the Neumann data. Lastly, in Section \ref{Conclusion} we provide a summary of the results of this paper and briefly discuss potential directions for future research.

%%%%%%%%%%%%%%%%%%%%%%%%%%%%%%%%%%%%%%%%%%%%%%%%%%%%%%%%%%%
\section{An Application to Diffuse Optical Tomography}\label{dot}\label{DOT-section}
We begin by considering the direct problem associated with DOT. This problem stems from semiconductor theory where boundary measurements are used to determine the existence of an interior structure. Recall, that we are concerned with the case where these interior structures are of small volume. We assume that the domain $\Omega \subset \mathbb{R}^{d}$ (for $d=2,3$) is a bounded simply connected open set with Lipschitz boundary $\partial \Omega$ with unit outward normal $\nu$. We let $D \subset \Omega$ with Lipschitz boundary $\partial D$ satisfying \eqref{sball}. 

Now, we let $u \in H^{1}(\Omega)$ be the unique solution to
\begin{equation}\label{dot-mpde}
- \Delta u + \rho \chi_D u= 0 \quad \text{in} \quad \Omega \quad \text{and} \quad u \big \rvert_{\partial \Omega} = f
\end{equation}
for any given $f \in H^{1/2} ( \partial \Omega )$ where $\chi_{(\cdot)}$ denotes the indicator function. We assume the absorption coefficient $\rho \in L^{\infty}(D)$. For analytical purposes of well-posedness for the direct problem and the upcoming analysis of the inverse problem, we assume that there are constants $\rho_{\text{min}}$ and $\rho_{\text{max}}$ such that
\begin{center}
$0 < \rho_{\text{min}} \leq \rho \leq \rho_{\text{max}}$ for a.e. $x \in D$.
\end{center}
One may easily verify that \eqref{dot-mpde} is well-posed by considering its variational formulation (see for e.g. \cite{evans}). Thus, one can show that for some $C>0$ that is independent of $0<\epsilon \ll 1$ we have that 
$$\|{u}\|_{H^1 (\Omega)} \leq C \|{f}\|_{H^{1/2}(\partial \Omega)}.$$ 
By equation \eqref{dot-mpde} we have that the Cauchy data is such that $(f , \partial_{\nu} u) \in H^{1/2}(\partial \Omega) \times H^{-1/2}(\partial \Omega)$.

In this section, we will develop the MUSIC Algorithm for solving the inverse problem under consideration. The goal is to first derive an asymptotic expansion for the Neumann data. Being motivated by analysis in \cite{MUSIC-kirsch, MUSIC-armin}, we will derive an analog of the multi-static response matrix derived from the reciprocity gap functional \eqref{rgf} for this problem.

\subsection{\textbf{MUSIC Algorithm}}

We begin, by proving an asymptotic expansion of the Neumann data $\partial_\nu u$ on $\partial \Omega$ in terms of the parameter $0<\epsilon \ll 1$. To this end, we let $u_0 \in H^{1}(\Omega)$ be the harmonic lifting of the Dirichlet data such that 
\begin{equation}\label{dot-hl}
- \Delta u_0 = 0 \quad \text{in} \quad \Omega \quad \text{and} \quad u_0 \big \rvert_{\partial \Omega} = f.
\end{equation}
In other words, $u_0$ satisfies the background problem associated with \eqref{dot-mpde} without the absorption coefficient with the same Dirichlet data $f \in H^{1/2}(\partial \Omega)$. We continue by defining the Dirichlet Green's function for the negative Laplacian on the known domain $\Omega$ as $\mathbb{G}(\cdot , z) \in H^{1}_{loc} (\Omega \setminus \brac{z})$, which is the unique solution to the boundary value problem
\[ - \Delta \mathbb{G}(\cdot , z) = \delta (\cdot , z) \enspace \text{in} \enspace \Omega \quad \text{and} \quad \mathbb{G}(\cdot , z) \big \rvert_{\partial \Omega} = 0.
\]
For any fixed $z \in \Omega$, we appeal to Green's 2nd Theorem to obtain the representation 
\begin{align*}
	 -(u - u_0)(z) = \int_{\Omega} (u - u_0)(x) \Delta \mathbb{G}(x , z) \, \text{d}x&= \int_{\Omega} \mathbb{G}(x,z) \rho (x) \chi_{D} (x) u(x) \, \text{d}x \\
	 &= \int_{D} \mathbb{G}(x,z) \rho (x) u(x) \, \text{d}x
\end{align*}
where we have used the fact that the absorption coefficient is zero outside of the region $D$. By taking the normal derivative, we have that for all $z \in \partial \Omega$ 
\begin{align}\label{dot-aexp}
\partial_{\nu} (u - u_0 )(z) & = - \int_{D} \rho(x) u(x) \partial_{\nu (z)} \mathbb{G} (x , z) \, \text{d}x \nonumber \\
					    &\hspace{-0.3in}= - \int_{ D} \rho(x) u_0(x) \partial_{\nu (z)} \mathbb{G} (x , z) \, \text{d}x - \int_{D} \rho(x) (u - u_0 )(x) \partial_{\nu (z)} \mathbb{G} (x , z) \, \text{d}x  
\end{align}
where the integrands are well defined due to the fact that $z \in \partial \Omega$ and $\partial_{\nu (z)}$ denotes the normal derivative on $\partial \Omega$ with respect to $z$. Given that the region $D$ satisfies \eqref{sball}, we claim that \eqref{dot-aexp} is dominated by the first integral. In other words, the Neumann data can be approximated by the harmonic lifting $u_0$ restricted to the small volume subregions, instead of the unknown photon density $u$.

The following estimate derived in Theorem 3.1 of \cite{cakoni2} will help us in our asymptotic analysis of \eqref{dot-aexp}. It states that for all $\varphi \in H^{1}(\Omega)$ with $D \subset \Omega$ such that $|D| = \mathcal{O}(\epsilon^d)$, we have that 
\begin{equation}\label{sobo}
\|{\varphi}\|_{L^{2}(D)} \leq C \epsilon^{\frac{d}{2} \big( 1 - \frac{2}{p} \big)} \|{\varphi}\|_{H^{1}(\Omega)}
\end{equation} 
where $p \geq 2$ in $d=2$ and $2 \leq p \leq 6$ in $d=3$. This estimate is proven using the Sobolev embedding of $H^{1}(\Omega) \hookrightarrow L^{p}(\Omega)$ (see for e.g. Chapter 5 of \cite{adams}). Using \eqref{sobo}, we prove that $u_0$ approximates $u$ when $D$ has small volume.
\begin{lemma} \label{dot-u0-approx}
For all $f \in H^{1/2}(\partial \Omega)$, let $u$ and $u_0$ be the solutions to \eqref{dot-mpde} and \eqref{dot-hl}, respectively. Then, we have that 
$$ \|{u - u_0}\|_{H^{1}(\Omega)} \leq C \epsilon^{d \big ( 1 - \frac{2}{p} \big )} \|{f}\|_{H^{1/2}(\partial \Omega)}$$
provided that $|D| = \mathcal{O}(\epsilon^d)$ where $p \geq 2$ in $d=2$ and $2 \leq p \leq 6$ in $d=3$.
\end{lemma}
\begin{proof} Since, $u-u_0 \in H_{0}^{1} (\Omega)$ we have that $\| u-u_0 \|_{H^1(\Omega)} \leq C  \| \nabla (u-u_0) \|_{L^2(\Omega)}$ by the Poincar\'{e} inequality. By appealing to equations \eqref{dot-mpde} and \eqref{dot-hl} along with Green's 1st Theorem, we have that 
\begin{align*}
 \int_{\Omega} \big | \nabla (u - u_0 ) \big |^2 \, \text{d}x  = - \int_{D} \rho u \overline{(u - u_0)} \, \text{d}x &\leq \rho_{\text{max}} \|{u}\|_{L^{2}(D)} \|{u - u_0}\|_{L^{2}(D)}.
\end{align*} 
Now, by using the estimate in \eqref{sobo}, we obtain that 
\begin{align*}
	 \|{u}\|_{L^{2}(D)} \|{u - u_0}\|_{L^{2}(D)} &\leq C \epsilon^{d\left( 1- \frac{2}{p} \right)} \|{u}\|_{H^{1} (\Omega)} \|{u - u_0}\|_{H^{1} (\Omega)}.
\end{align*}
This proves the claim by appealing to the well-posedness of \eqref{dot-mpde}.
\end{proof}
From the above lemma, we have shown that $u$ can be approximated by $u_0$ in norm when $|D| = \mathcal{O}(\epsilon^d)$ is small. Under the same assumption, we will use the previous lemma along with \eqref{sobo} to compare the magnitudes of the two integrals in equation \eqref{dot-aexp}. We start by analyzing the second integral and provide the following results.
\begin{lemma}\label{dot-2nd-int}
For $z \in \partial \Omega$ and $|D| = \mathcal{O}(\epsilon^d)$, we have that $$\int_{D} \rho (x) (u - u_0)(x) \partial_{\nu (z)} \mathbb{G}(x,z) \, \text{d}x = \mathcal{O}(\epsilon^{d+1}) \quad \text{as} \quad \epsilon \rightarrow 0.$$
\end{lemma}
\begin{proof}
In order to prove the claim, we must estimate 
\begin{align*}
	\left| \int_{D} \rho (x) (u - u_0 ) (x) \partial_{\nu (z)} \mathbb{G} (x , z) \, \text{d}x  \right| 
	&\leq C \epsilon^{\frac{d}{2}\left( 1- \frac{2}{p} \right)} \|{u - u_0}\|_{H^{1} (\Omega)} \|{\partial_{\nu (z)} \mathbb{G}(\cdot , z )}\|_{L^{2} (D)} \\
	&\leq C \epsilon^{\frac{3d}{2}\left( 1- \frac{2}{p} \right)} \|{f}\|_{H^{1/2}(\partial \Omega)} \|{\partial_{\nu (z)} \mathbb{G}(\cdot , z )}\|_{L^{2} (D)}
\end{align*} 
where we have used \eqref{sobo}, and Lemma \ref{dot-u0-approx} in order. We also have that 
\begin{align*}
	\|{\partial_{\nu (z)} \mathbb{G}(\cdot , z )}\|_{L^{2} (D)} \leq  \|{\nabla \mathbb{G}(\cdot , z )}\|_{L^2(D)} \leq C \epsilon^{d/2} \|{\mathbb{G}(\cdot , z )}\|_{C^{1} ( \Omega^{*})}
\end{align*}
by the fact that $\nu$ is a unit vector and the symmetry of the Green's function. The region $\Omega^{*}$ satisfies that $D \subset \Omega^{*} \subset \Omega$ for all $0 < \epsilon \ll 1$ with dist$(\partial \Omega , \Omega^{*})>0$. Thus, we have that for all $z \in \partial \Omega$ 
\begin{align}\label{powereps}
\left| \int_{D} \rho (x) (u - u_0 )(x) \partial_{\nu (z)} \mathbb{G} (x , z) \, \text{d}x \right|  \leq C \epsilon^{d \left( 2 - \frac{3}{p} \right)} \|{f}\|_{H^{1/2}(\partial \Omega)}.
\end{align}
For $d=2$, we recall that $ p \geq 2$. In order to prove the claim, we impose the condition that 
$$3 = 2 \left( 2 - \frac{3}{p} \right) \quad \text{making the exponent of $\epsilon$ equal to $d+1$ in \eqref{powereps}},$$ 
which yields that $p=6$. From the above inequality we get that  
$$ \int_{D}\rho (x) (u - u_0 )(x) \partial_{\nu (z)} \mathbb{G} (x , z)\, \text{d}x \leq  C \epsilon^{3} \|{f}\|_{H^{1/2}(\partial \Omega)}.$$
Similarly, for $d=3$, we recall that $2 \leq p \leq 6$. 
Again, to prove the claim we impose that 
$$4 = 3 \left(2 - \frac{3}{p} \right) \quad \text{again making the exponent of $\epsilon$ equal to $d+1$ in \eqref{powereps}},$$ 
which yields that $p=4.5$. Thus, we have that  
$$ \int_{ D} \rho (x) (u - u_0 )(x) \partial_{\nu (z)} \mathbb{G} (x , z) \, \text{d}x \leq  C \epsilon^{4} \|{f}\|_{H^{1/2}(\partial \Omega)}.$$
Therefore, for both $d=2$ and $d=3$ taking $p=6$ and $p=4.5$, respectively, we have that 
$$ \int_{D} \rho (x) (u - u_0 )(x) \partial_{\nu (z)} \mathbb{G} (x , z) \, \text{d}x  = \mathcal{O} \big( \epsilon^{d+1} \big) \quad \text{as} \quad \epsilon \to 0$$
which proves the claim. 
\end{proof}
Next, we show that the first integral in \eqref{dot-aexp} is of order $\epsilon^{d}$. From this, equation \eqref{dot-aexp} will imply that the first integral is the leading term, rendering the second integral as negligible. This is proven in the following result. 

\begin{lemma}\label{dot-1st-int}
For all $z \in \partial \Omega$ where $D$ is given by \eqref{sball}, we have that as $\epsilon \rightarrow 0$ $$ \int_{D} \rho (x) u_0 (x) \partial_{\nu (z)} \mathbb{G}(x,z) \, \text{d}x = \epsilon^d \sum_{j=1}^{J} |B_j| \text{Avg}(\rho_j) u_0 (x_j ) \partial_{\nu (z)} \mathbb{G}(x_j , z) + \mathcal{O}(\epsilon^{d+1})$$ where Avg$(\rho_j)$ is the average value of $\rho$ in $D_j$.
\end{lemma}
\begin{proof}
By \eqref{sball}, we have that $x \in D_j$ if and only if $x = x_j + \epsilon y$ for some $y \in B_j$. Now, recall that both $u_0$ and $\partial_{\nu (z)} \mathbb{G}(\cdot ,z)$ are smooth in the interior of of $\Omega$ since $z \in \partial \Omega$ by standard elliptic regularity. Therefore, we have that for all $x \in D_j$ 
$$ u_0 (x) \partial_{\nu (z)} \mathbb{G}(x,z) = u_0 (x_j + \epsilon y) \partial_{\nu (z)} \mathbb{G}(x_j + \epsilon y,z) = u_0 (x_j) \partial_{\nu (z)} \mathbb{G}(x_j,z) + \mathcal{O}(\epsilon)$$ 
as $\epsilon \rightarrow 0$ by appealing to Taylor's Theorem. From this, we obtain that 
\begin{align*}
	\int_{D} \rho (x) u_0 (x) \partial_{\nu (z)} \mathbb{G} (x , z) \, \text{d}x  &= \sum_{j=1}^{J} \int_{D_j} \rho (x) u_0 (x_j + \epsilon y) \partial_{\nu (z)} \mathbb{G}(x_j + \epsilon y,z) \, \text{d}x \\
	&= \sum_{j=1}^{J} \Big ( u_0 (x_j) \partial_{\nu (z)} \mathbb{G}(x_j,z) + \mathcal{O}(\epsilon) \Big )  \int_{D_j} \rho (x) \, \text{d}x.
\end{align*} 
This implies that 
$$\int_{D} \rho (x) u_0 (x) \partial_{\nu (z)} \mathbb{G} (x , z) \, \text{d}x  = \epsilon^d \sum_{j=1}^{J} |B_j| \text{Avg}(\rho_j) u_0 (x_j ) \partial_{\nu (z)} \mathbb{G}(x_j , z) + \mathcal{O}(\epsilon^{d+1})$$ 
as $\epsilon \rightarrow 0$ where $\text{Avg}(\rho_j)$ denotes the average value of $\rho$ in $D_j$ as well as using the fact that $|D_j| = \epsilon^d |B_j|$.
\end{proof}
Using Lemmas \ref{dot-2nd-int} and \ref{dot-1st-int}, it is clear that for a specified $z \in \partial \Omega$, the normal derivative of the difference of $u$ and $u_0$ is dominated by the first integral from equation \eqref{dot-aexp}. Therefore, we have proven an asymptotic expansion for the boundary data $\partial_\nu u$ in therms of the known harmonic lifting and Green's function. Similar results have been proven in \cite{MUSIC-ammari-eit,MUSIC-Hanke} using boundary integral operators.

\begin{theorem}\label{dot-main-thm}
For any $z \in \partial \Omega$ we have that $$\partial_{\nu (z)} u (z) = \partial_{\nu (z)} u_0 (z) - \epsilon^d \sum_{j=1}^{J} u(x_j) \text{Avg}(\rho_j) |B_j| \partial_{\nu (z)} \mathbb{G}(x_j , z) + \mathcal{O}(\epsilon^{d+1}) \quad \text{as} \quad \epsilon \rightarrow 0$$
provided that $D$ satisfies \eqref{sball}.
\end{theorem}

We use this asymptotic expansion to develop an algorithm that detects the centers of the defective regions. To achieve this, we study the MUSIC algorithm which can be considered as a discrete analogue of the factorization method (see for e.g. \cite{cheney1,kirschbook,MUSIC-kirsch}). To this end, we let $u_0(\cdot ,g)$ and $u_0(\cdot ,f )$ denote the harmonic liftings with Dirichlet data  $g$ and $f \in H^{1/2}(\partial \Omega)$, respectively (see for e.g. Chapter 2 of \cite{evans}). Thus, using Theorem \ref{dot-main-thm} we can approximate the reciprocity gap functional \eqref{rgf} with Cauchy data $(u=f , \partial_{\nu} u)$ on $\partial \Omega$ and input $v=u_0(\cdot ,g)$. Therefore, we have that
\begin{align*}
R_f \big[ u_0( \cdot ,g) \big ]  &=  \int_{\partial \Omega} u_0 (\cdot , g) \partial_{\nu} u(\cdot , f) - u(\cdot , f) \partial_{\nu} u_0 (\cdot , g) \, \text{d}s \\
	&= \int_{\partial \Omega}  u_0(\cdot , g) \Big [ \partial_{\nu} u_0(\cdot , f) - \epsilon^d \sum_{j=1}^{J} u(x_j) \text{Avg}(\rho_j) |B_j| \partial_{\nu (z)} \mathbb{G}(x_j , z) \Big ] \, \text{d}s\nonumber\\
    & \phantomrel{=} {} -  \int_{\partial \Omega} u_0(\cdot , f) \partial_{\nu} u_0(\cdot ,g)\, \text{d}s  + \mathcal{O}(\epsilon^{d+1}) \\
	&=  - \epsilon^{d} \sum_{j=1}^{J} |B_j| \text{Avg}(\rho_j) u_0(x_j ,f ) \int_{\partial \Omega} u_0(\cdot ,g) \partial_{\nu (z)} \mathbb{G}(x_j , z) \, \text{d}s + \mathcal{O}(\epsilon^{d+1})
\end{align*} 
where we used the fact that $u \big \rvert_{\partial \Omega} = u_0(\cdot ,f) \big \rvert_{\partial \Omega} = f$, as well as $u_0(\cdot ,f)$ and $u_0(\cdot ,g)$ being harmonic in $\Omega$.  Furthermore, since $z \in \partial \Omega$, we have that 
$$\int_{\partial \Omega} u_0(\cdot ,g) \partial_{\nu (z)} \mathbb{G}(x_j , z) \, \text{d}s = - u_0(x_j,g).$$ 
With this, we obtain the expansion  
\begin{equation}\label{dot-rgf-approx}
R_f \big[ u_0(\cdot ,g) \big] = \epsilon^d \sum_{j=1}^{J} |B_j| \text{Avg}(\rho_j) u_0(x_j,f) u_0(x_j ,g) + \mathcal{O}(\epsilon^{d+1})
\end{equation}
as $\epsilon \rightarrow 0$ where we have made the dependance on $f$ explicit. 

In order to derive the MUSIC algorithm, {\it we assume that $\Omega$ is the unit circle for $d=2$} where we let $g = \text{e}^{\text{i}m \theta}$ and $f = \text{e}^{\text{i}n \theta}$ for $m,n = 0, \cdots , N$ for some fixed $N \in \N$. Here $\theta$ denotes the angle formed by points on $\partial \Omega$ when converted to polar coordinates. Using only the leading order term of \eqref{dot-rgf-approx}, we define the matrix 
$$ \textbf{F}_{n,m} = \epsilon^d \sum_{j=1}^{J} |B_j| \text{Avg}(\rho_j) u_0(x_j,f_n) u_0(x_j,g_n).$$
From the definition of \textbf{F}, we see that it can be factorized by the matrices $\textbf{U} \in \mathbb{C}^{(N+1)\times J}$ and $\textbf{T} \in \mathbb{C}^{J \times J}$ that are given by $$ \text{\textbf{U}}_{m,j} = u_0(x_j , f_m) \quad \text{ and } \quad \text{\textbf{T}} = \text{diag} \big(\epsilon^{d} |{B_j}| \text{Avg}(\rho_{j}) \big).$$  
Therefore, it is easy to see that \textbf{F}=\textbf{UTU}$^\top$. Notice, that by our assumptions on $\rho$, all the diagonal entries of the matrix \textbf{T} are non-zero. We now define the vector $\boldsymbol{\phi}_x \in \mathbb{C}^{N+1}$ for any $x \in \mathbb{R}^{d}$ by
\begin{equation}\label{phix}
\boldsymbol{\phi}_x = \left( u_0(x,f_0), \cdots , u_0(x,f_N) \right)^{\top }.
\end{equation}
The goal is to prove that the vector $\boldsymbol{\phi}_x $ is in the range of $\text{\textbf{FF}}^{*}$ if and only if $x$ is contained in the set $\brac{x_j : j = 1 , \hdots , J}$ as similarly done in \cite{EIT-granados1}. This is a discretized version of the factorization method initially studied for this problem in \cite{harris1}. See for e.g. \cite{cheney1,MUSIC-kirsch} for the connection of the factorization method and MUSIC algorithm. 

We now construct an imaging functional derived from the leading order term in the asymptotic expansion of the reciprocity gap functional. To this end, we need to show that for each sampling point $x \in \Omega$ we have that $\boldsymbol{\phi}_x $ is in the range of \textbf{FF}$^{*}$ if and only if $x \in \{ x_j : j = 1 , \hdots , J\}$. This result has been proven in Theorem 3.2 of \cite{EIT-granados1}. To avoid repetition, we will state the following result and reference the proofs in Section 3 of \cite{EIT-granados1} for details. 

\begin{theorem}\label{musicthm}
Assume that $N+1 > J$. Then for all $x \in \Omega$ being given by the unit circle 
$$\boldsymbol{\phi}_x \in Range ( \text{\textbf{FF}}^{*} ) \quad \text{if and only if} \quad x \in \brac{x_j : j = 1 , \hdots , J}$$ 
where $\boldsymbol{\phi}_x$ is defined as in \eqref{phix}. Moreover, the rank of the matrix $\text{\textbf{FF}}^{*} $ is given by $J$. 
\end{theorem}

Notice, that the matrix \textbf{F} can be approximated by the known reciprocity gap functional. This implies that Theorem \ref{dot-main-thm} can be used to recover the centers of the subregions $\{x_j : j = 1 , \hdots ,J\}$. To this end, we must verify whether $\boldsymbol{\phi}_x \in$ Range$(\text{\textbf{FF}}^{*})$. This is equivalent to $\text{\textbf{P}}\boldsymbol{\phi}_x$ = 0 where $\text{\textbf{P}}$ is the orthogonal projection onto the Null$(\text{\textbf{FF}}^{*})$.

%%%%%%%%%%%%%%%%{%%%%%%%%%%%%%%%%%%%%%%%%%%%%%%%%%%%%%%%%%%%
\subsection{Numerical Validation for the MUSIC Algorithm}
We now provide some numerical examples of recovering locations the $\{x_j : j = 1 , \hdots ,J\}$ using Theorem \ref{musicthm}. All of our numerical experiments are done with the software \texttt{MATLAB} 2020a. We will let $\Omega$ be given by the unit circle in $\R^2$ and we need to compute the Neumann data $\partial_{\nu} u$. It is clear that the Neumann data can be approximated by the harmonic lifting $u_0$ with the same Dirichlet data. Indeed, Lemmas \ref{dot-2nd-int} and \ref{dot-1st-int} imply that for all $z \in \partial \Omega$
$$\partial_{\nu} u \approx \partial_{\nu} u_0 - \int_{D} \rho(x) u_0(x) \partial_{\nu (z)} \mathbb{G} (x , \cdot) \, \text{d}x.$$
This can be seen as an analog to the Born approximation used in scattering theory (see for e.g. \cite{MUSIC-kirsch}). It is clear that  for Dirichlet data $\text{e}^{\text{i}n \theta}$ the harmonic lifting is given by 
$$u_0(x, \text{e}^{\text{i}n \theta}) = |x|^n \text{e}^{\text{i}n \theta} \quad \text{ for all } \quad n \in \mathbb{N}\cup \{ 0 \}.$$ 
It is also well known that the normal derivative of $ \mathbb{G} (x , z)$ is given by
$$ \partial_{\nu(z)} \mathbb{G}\big ( x  , z \big ) \big|_{\partial \Omega} =  \frac{1}{2 \pi} \left[ \frac{1 - |x|^2 }{|x|^2 + 1 - 2 |x| \text{cos}(\theta - \theta_{z})}   \right]$$
for $z \in \partial \Omega$. Therefore, we can compute $\partial_{\nu} u$ using the `\texttt{integral2}' command in \texttt{MATLAB}. 

Given the Dirichlet data 
$$f_m = \text{e}^{\text{i}m \theta} \quad \text{ and its corresponding Neumann data} \quad  \partial_{\nu} u (\cdot ,\text{e}^{\text{i}m \theta}),$$
 we can easily approximate the reciprocity gap functional as given by \eqref{rgf} using 64 equally spaced points on the unit circle for $v=u_0 ( \cdot ,\text{e}^{\text{i}n \theta}) $. We have that 
$$R_{f_m} \big[u_0 ( \cdot ,\text{e}^{\text{i}n \theta})\big] = \int_{\partial \Omega} \text{e}^{\text{i}n \theta} \partial_{\nu}u(\cdot , \text{e}^{\text{i}m \theta}) - u(\cdot , \text{e}^{\text{i}m \theta}) n \text{e}^{\text{i}n \theta}\, \text{d}s$$ 
which is approximated via a Riemann sum using the `\texttt{dot}' command in \texttt{MATLAB} for $m,n = 0 , \cdots , 20$. By appealing to the asymptotic result in Theorem \ref{dot-main-thm}, we have that 
$$\textbf{F}_{n,m} \approx \int_{\partial \Omega} \text{e}^{\text{i}n \theta} \partial_{\nu}u(\cdot , \text{e}^{\text{i}m \theta}) - u(\cdot , \text{e}^{\text{i}m \theta}) n \text{e}^{\text{i}n \theta}\, \text{d}s\quad \text{ such that } \quad \textbf{F} \in \C^{21 \times 21}.$$

Once $\textbf{F}$ has been approximated we can use Theorem \ref{musicthm} to recover the locations of the components of $D$. We only need to check if the vector $\boldsymbol{\phi}_x$ is in the range of $\text{\textbf{FF}}^{*}$. Therefore, we compute the norm 
$$\| \text{\textbf{P}}\boldsymbol{\phi}_x \|^2_2 = \sum\limits_{\ell=r+1}^{21} \left|\big( \boldsymbol{\phi}_x , {\bf u}_\ell \big) \right|^2 $$
where the vectors ${\bf u}_\ell$ are the orthonormal eigenvectors for ${\bf F}{\bf F}^*$ and $r=$Rank$\big( {\bf F}{\bf F}^* \big)$. 
Recall, that the vector 
$$\boldsymbol{\phi}_x = \bigg(1,  |x| \text{e}^{\text{i} \theta}\,  , \, \hdots \, , \, |x|^{20} \text{e}^{20 \text{i}\theta} \bigg)^\top$$ 
by equation \eqref{phix} where $\theta$ is the polar angle for the sampling point $x \in \Omega$. Here $\text{\textbf{P}}$ denotes the orthogonal projection onto the Null$(\text{\textbf{FF}}^{*})$. Therefore, the imaging functional for recovering the centers is given by 
$$W_{\text{MUSIC}}(x) = \left[ \sum\limits_{\ell=r+1}^{21} \left|\big( \boldsymbol{\phi}_x , {\bf u}_\ell \big) \right|^2 \right]^{-1} \quad \text{ for any  } \quad x \in \Omega$$
which has the property that $W_{\text{MUSIC}}(x)\gg1$ for $x=x_j$ and $W_{\text{MUSIC}}(x)= \mathcal{O}(1)$ for $x \neq x_j$. We will plot the imaging functional to provided a numerical approximation of the centers $x_j$ for $j=1 , \cdots , J$.

In Figures \ref{dot-music0}, \ref{dot-music1}, and \ref{dot-music2} we use the imaging functional $W_{\text{MUSIC}}(x)$ given above to recover the locations of the two components of the region $D$. In theses experiments, the region 
$$D =\big( x_1 + \epsilon B(0,1) \big) \bigcup  \big( x_2 + \epsilon B(0,1)\big)$$
with $B(0,1)$ being the unit ball centered at the origin. The points $x_1$ and $x_2$ are points contained in the region $\Omega$. We will take the forcing term to be given by $\rho=1$ on both components of $D$. Here, we take $\epsilon = 0.01$ as well as adding $\delta = 5\%$ random noise to the computed normal derivative of the difference of $u$ and its harmonic lifting to simulate error in measured data. \\

\noindent{\bf Example 1: }\\
In our first example presented here, we let
$$x_1 = (-0.25 , 0.25) \quad \text{and} \quad x_2 = (0.25 , -0.25)$$
for the reconstruction in Figure \ref{dot-music0}. Here we let $\delta = 5 \%$ and $\rho = 1$ in both subregions. Presented is a contour and surface plot of the imaging functional $W_{\text{MUSIC}}(x)$. As we can see, the imaging functional is elevated in the general region around the centers.
\begin{figure}[ht]
\centering 
\includegraphics[scale=0.15]{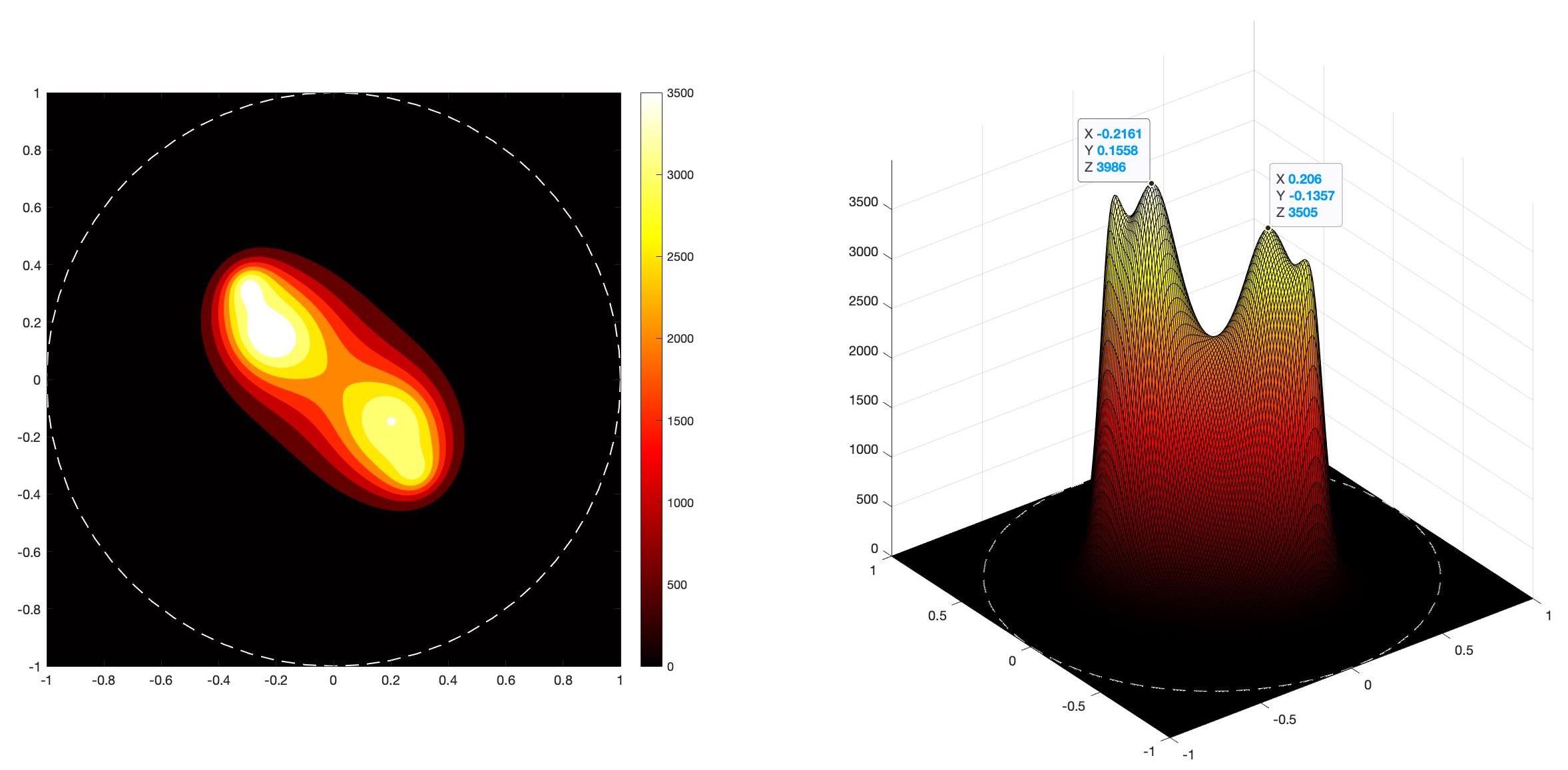}
\caption{Initial reconstruction of the locations $x_1 = (-0.25 , 0.25)$ and $x_2 = (0.25 , -0.25)$ via the MUSIC algorithm with the faulty rank of $\textbf{FF}^*$. Contour plot  on the left and Surface plot on the right of the imaging function $W_{\text{MUSIC}}(x)$.}
\label{dot-music0}
\end{figure}

Recall, that the imaging functional $W_{\text{MUSIC}}(x)$ depends on the rank of $\textbf{FF}^*$, which was calculated using the \texttt{rank} function in \texttt{MATLAB}. However, $\|{\textbf{FF}^*}\| \ll 1$ and the default tolerance of the \texttt{rank} function produces an overestimation of the true rank since the singular values of $\textbf{FF}^*$ are very small. For the rest of our numerical experiments, we improve the rank calculation by computing the singular values of $\textbf{FF}^*$ and ad hoc checking when a singular value decreases by at least 3 orders of magnitude from the previous one.
\begin{figure}[ht]
\centering 
\includegraphics[scale=0.18]{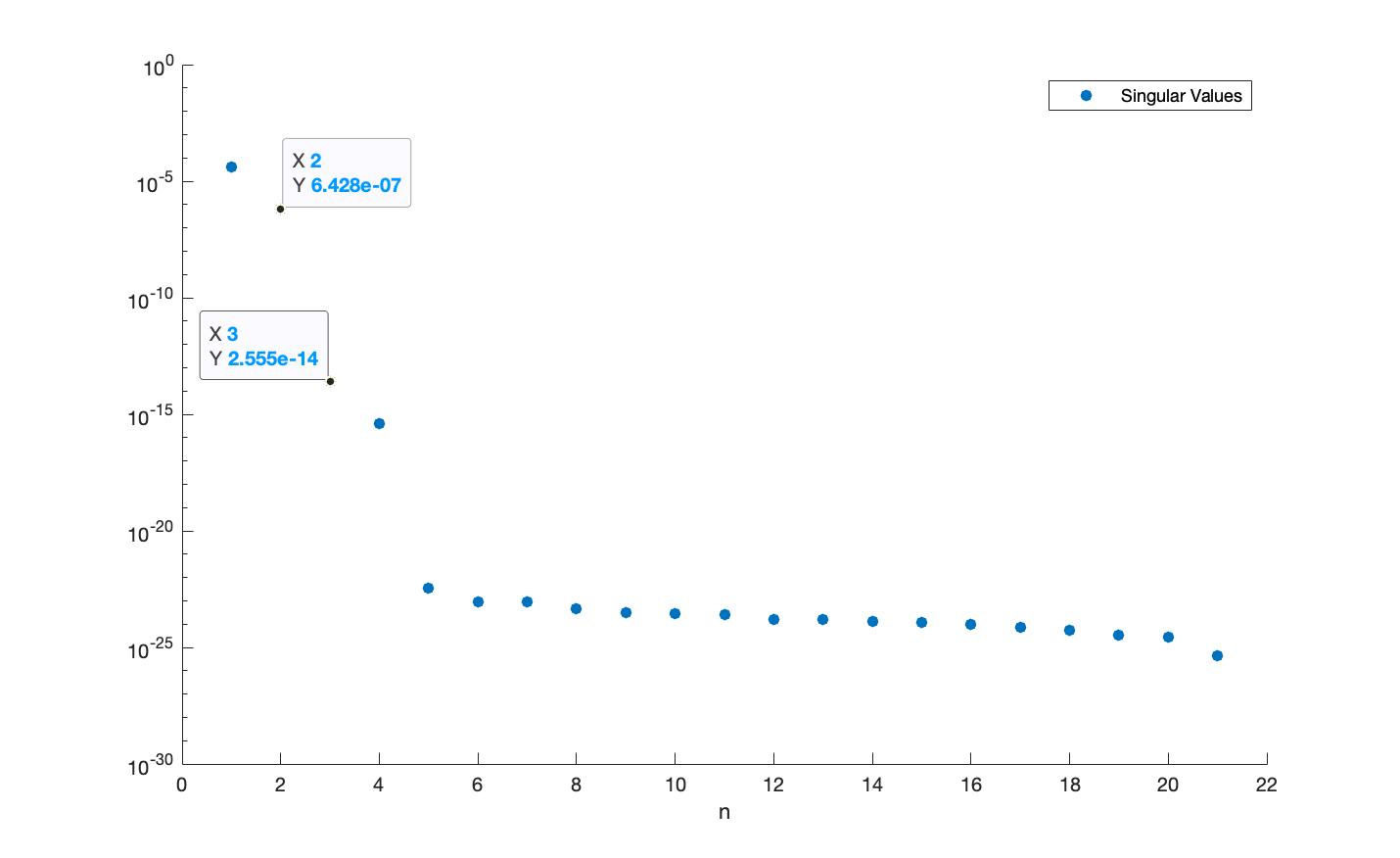}
\caption{Semi-log plot of the \texttt{svd}($\textbf{FF}^*$). Labeled are its 2nd and 3rd singular values.}
\label{sing}
\end{figure}

Figure \ref{sing}, suggests that the actual rank of $\textbf{FF}^*$ is 2, as expected by Theorem \ref{musicthm} since there are 2 components of $D$. Throughout the remaining examples of this section, we will continue to heuristically calculate the rank of $\textbf{FF}^*$ with this method. With this new method of calculating the rank, Figure \ref{dot-music1} demonstrates a clearer reconstruction of the 2 subregions centered at the locations  $x_1 = (-0.25 , 0.25)$ and $x_2 = (0.25 , -0.25)$. As we can see from the data tips, the improved imaging functional has spikes at the points
$$\widetilde{x}_1 = (-0.2462, 0.2462) \quad \text{and} \quad \widetilde{x}_2 = (0.2462, -0.2462).$$
Here we see that the locations $\widetilde{x}_1$ and $\widetilde{x}_2$ given by the MUSIC algorithm provide a good approximation for the actual locations of the components of the region D. \\
\begin{figure}[ht]
\centering 
\includegraphics[scale=0.15]{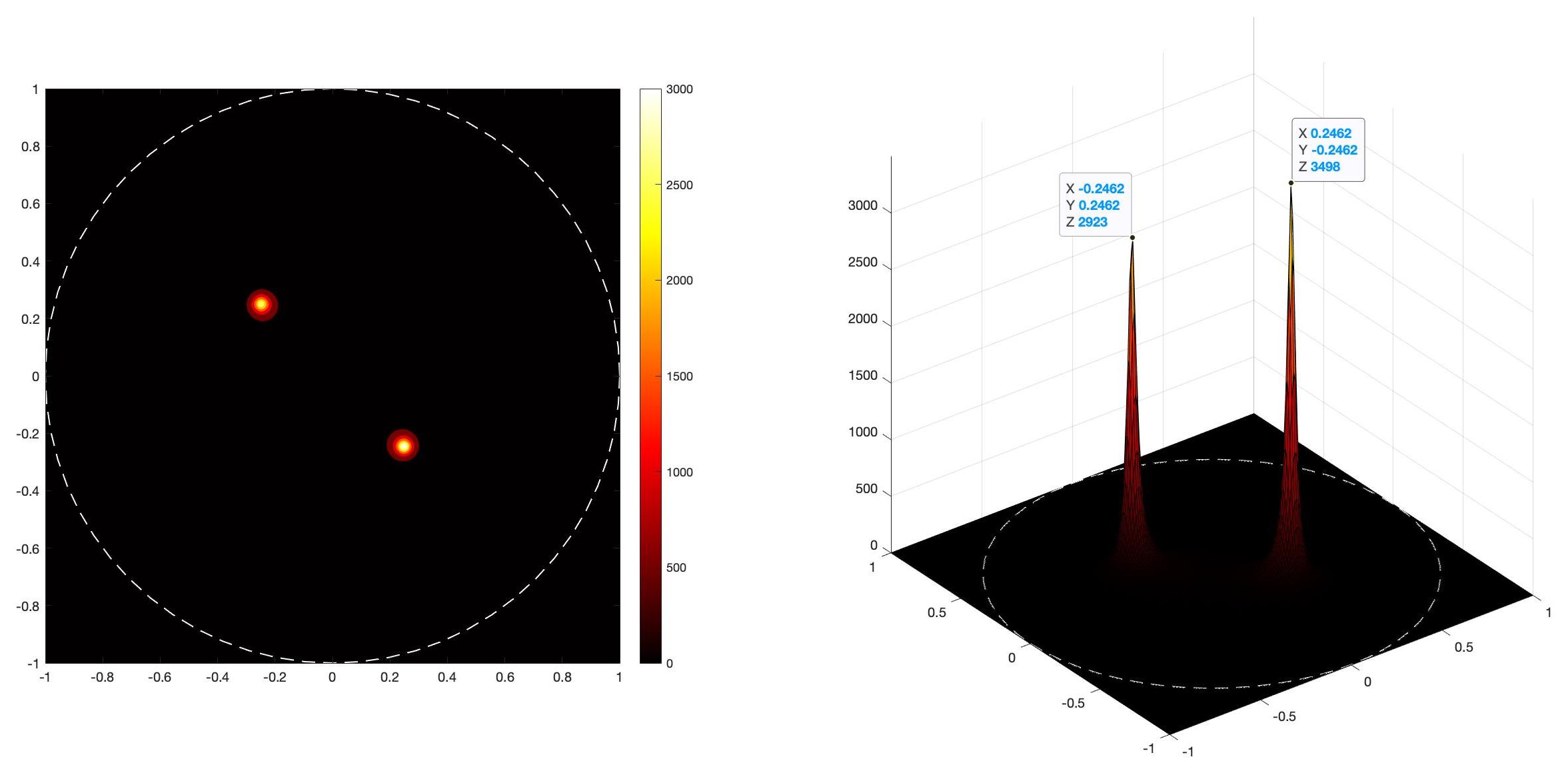}
\caption{Reconstruction of the locations  $x_1 = (-0.25 , 0.25)$ and $x_2 = (0.25 , -0.25)$ via the MUSIC algorithm. Contour plot  on the left and Surface plot on the right of the imaging functional   $W_{\text{MUSIC}}(x)$.}
\label{dot-music1}
\end{figure}

\noindent{\bf Example 2: }\\
In our second example presented here, we let 
$$x_1 = (-0.25 , 0.25) \quad \text{and} \quad x_2 = (-0.25 , -0.25)$$
for the reconstruction in Figure \ref{dot-music2}. Presented is a contour and surface plot of the imaging functional $W_\text{MUSIC}(x)$. We again let $\delta = 5\%$ and $\rho = 1$ in both subregions. As we can see from the data tips, the imaging functional has spikes at the points $$ \widetilde{x}_1 = (-0.2462, 0.2462) \quad \text{and} \quad \widetilde{x}_2 = (-0.2462, -0.2462).$$ Again, in this example we see that the locations of $\widetilde{x}_1 $ and $\widetilde{x}_2$ provide an approximation for the locations of the components of the region $D$. \\
\begin{figure}[ht]
\centering 
\includegraphics[scale=0.15]{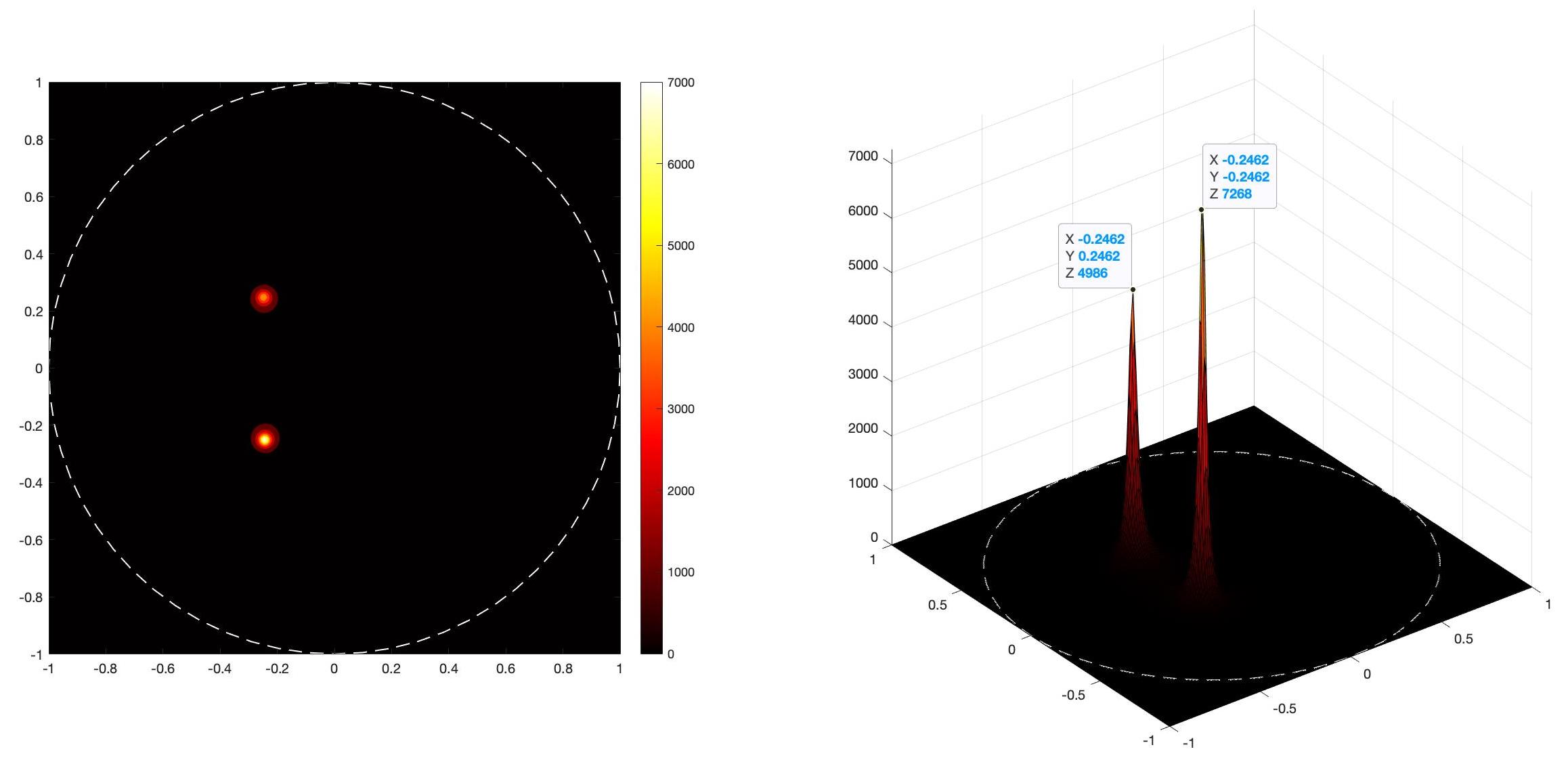}
\caption{Reconstruction of the locations  $x_1 = (-0.25 , 0.25)$ and $x_2 = (-0.25 , -0.25)$ via the MUSIC algorithm. Contour plot  on the left and Surface plot on the right of the imaging functional  $W_{\text{MUSIC}}(x)$.}
\label{dot-music2}
\end{figure}

\noindent
In our final two examples of this section, we similarly let the region 
$$D = \bigcup_{j=1}^J \big ( x_j + \epsilon B(0,1) \big) $$ 
where $J=3$ and $J=4$, respectively, with $B(0,1)$ being the unit ball centered at the origin. The points $x_j$ for $j=1, \cdots , J$ are contained in the region $\Omega$ and we once again let $\epsilon = 0.01$. However, we vary the value of the forcing term $\rho$ on each of the components of $D$. Furthermore, we add $\delta = 10 \%$ random noise to the approximated normal derivative of the difference of $u$ and its harmonic lifting to simulate error in measured data.\\

\noindent{\bf Example 3: }\\
In our third example presented here, we let
$$x_1 = (-0.75 , 0), \quad x_2 = (0.25 , 0.5), \quad \text{and} \quad x_3 = (-0.3, -0.4)$$
for the reconstruction in Figure \ref{dot-music3}. Presented is a contour and surface plot of the imaging functional $W_\text{MUSIC}(x)$. As we can see from the data tips, the imaging functional has spikes at the points 
$$ \widetilde{x}_1 = (-0.7487, -0.0151), \quad \widetilde{x}_2 = (0.2563, 0.4975), \quad \text{and} \quad \widetilde{x}_3 = (-0.2965, -0.407).$$ In this example we see that the locations of $\widetilde{x}_1 $, $\widetilde{x}_2$, and $\widetilde{x}_3$ provide an approximation for the locations of the components of the region $D$. For this example we let $\delta = 10\%$ where $\rho = 1/4$ in the region centered at $x_1$, $\rho = 1$ in the region centered at $x_2$, and $\rho = 2$ in the region centered at $x_3$. Notice that this example suggests that the MUSIC algorithm gives sharper reconstructions when the regions are well separated.\\
\begin{figure}[ht]
\centering 
\includegraphics[scale=0.15]{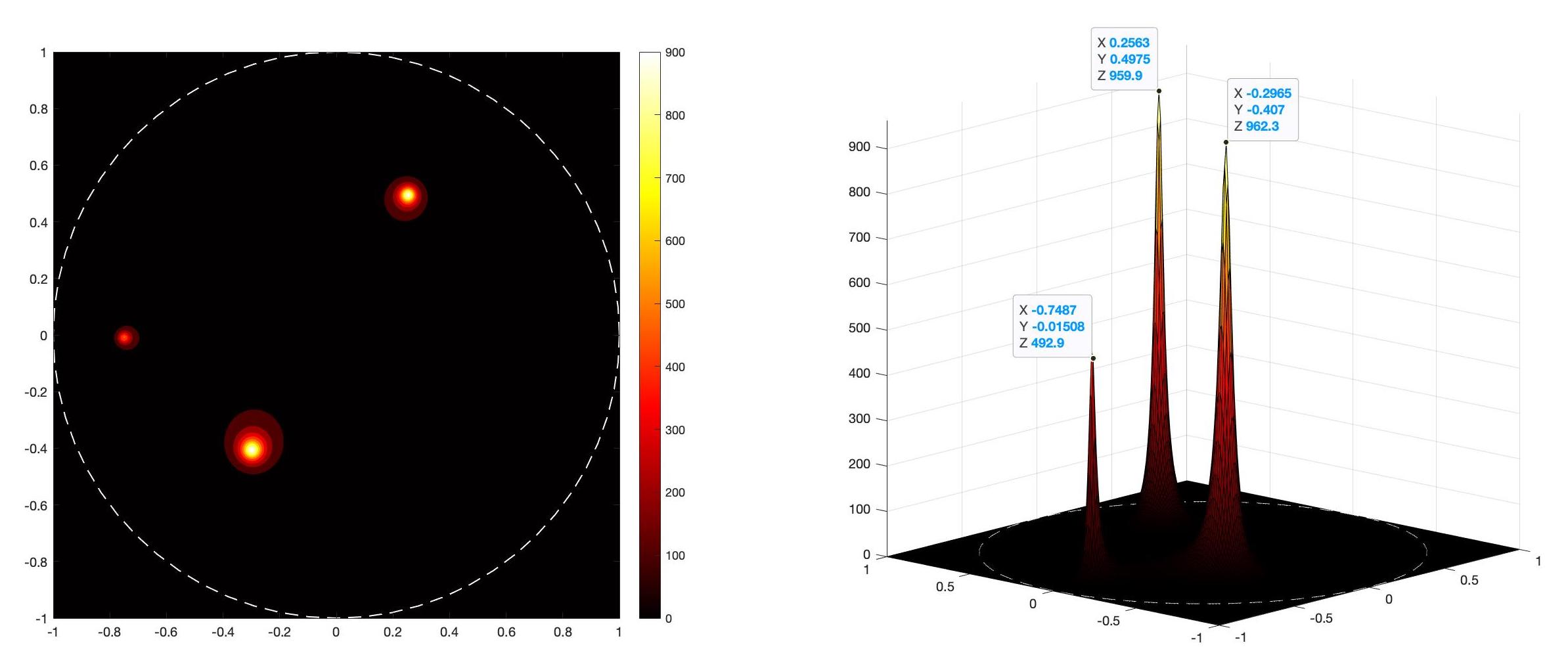}
\caption{Reconstruction of the locations $x_1 = (-0.75 , 0)$,  $x_2 = (0.25 , 0.5)$, \text{and}  $x_3 = (-0.3, -0.4)$ via the MUSIC algorithm. Contour plot on the left and Surface plot on the right of the imaging functional $W_{\text{MUSIC}}(x)$.}
\label{dot-music3}
\end{figure}

\noindent{\bf Example 4: }\\
In our final example presented here, we let
$$x_1 = (0 , 0.75), \quad x_2 = (-0.25 , 0.25), \quad x_3 = (0.25 , -0.25) \quad \text{and} \quad x_4 = (0.2 , -0.6)$$
for the reconstruction in Figure \ref{dot-music4}. Presented is a contour and surface plot of the imaging functional $W_\text{MUSIC}(x)$. As we can see from the data tips, the imaging functional has spikes at the points 
$$\widetilde{x}_1 = (-0.005, 0.7487), \quad \widetilde{x}_2 = (-0.2362, 0.2362), \quad \widetilde{x}_3 = (0.2362, -0.2562),$$
and 
$$\widetilde{x}_4 = (0.206, -0.5879).$$
In this example, we see that the reconstructed locations provide an approximation for the locations of the components of the region $D$. For this example, we let $\delta = 10\%$ where $\rho = 3/4$ in the region centered at $x_1$, $\rho = 1$ in the region centered at $x_2$, $\rho = 3/2$ in the region centered at $x_3$, and $\rho = 1/2$ in the region centered at $x_4$.\\
\begin{figure}[ht]
\centering 
\includegraphics[scale=0.15]{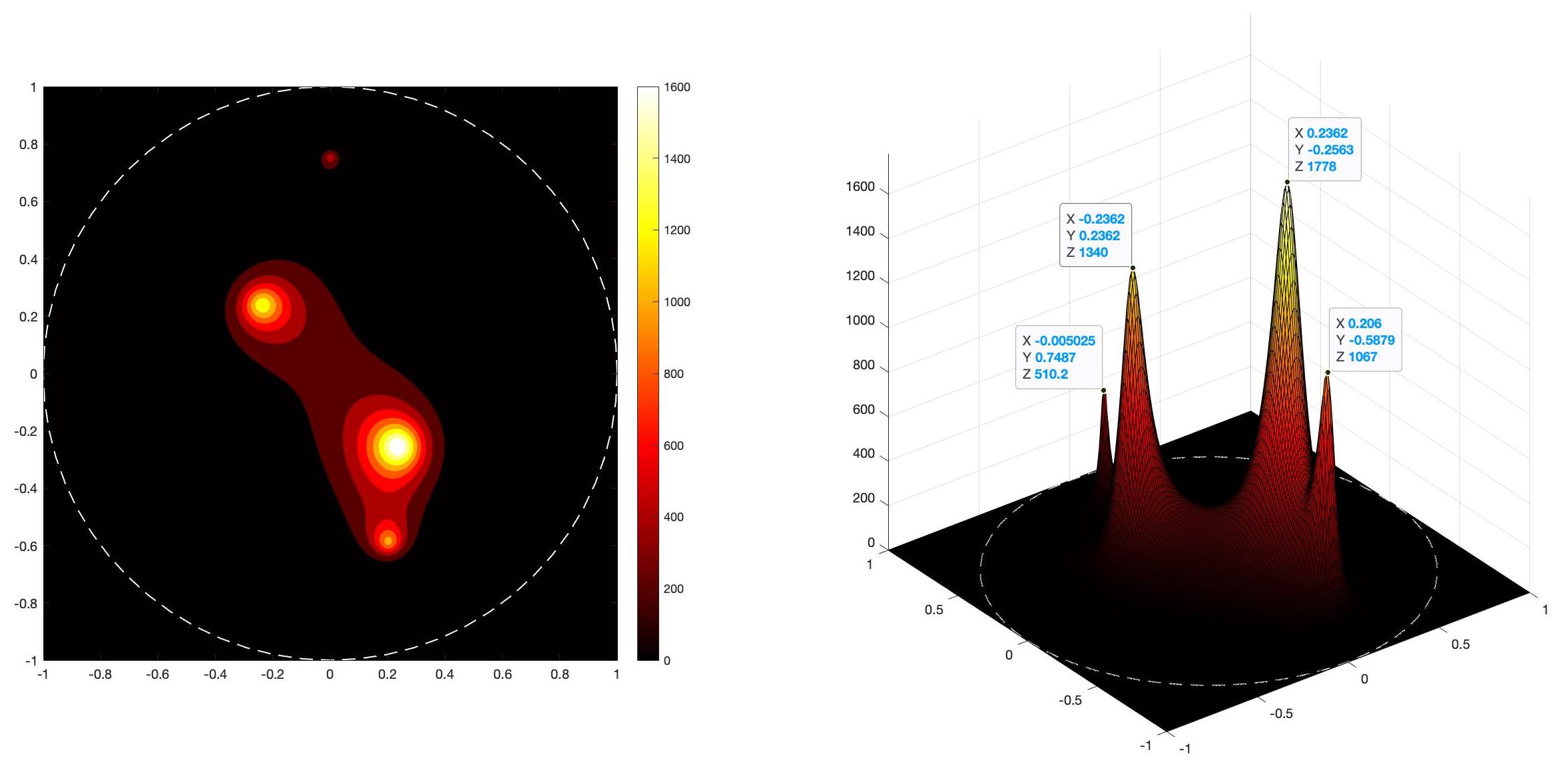}
\caption{Reconstruction of the locations $x_1 = (0 , 0.75)$, $x_2 = (-0.25 , 0.25),$ $x_3 = (0.25 , -0.25)$ \text{and} $x_4 = (0.2 , -0.6)$ via the MUSIC algorithm. Contour plot on the left and Surface plot on the right of the imaging functional $W_{\text{MUSIC}}(x)$.}
\label{dot-music4}
\end{figure}

%%%%%%%%%%%%%%%%{%%%%%%%%%%%%%%%%%%%%%%%%%%%%%%%%%%%%%%%%%%%
\section{\textbf{An Application to Inverse Scattering}}\label{InvScat-section}
We now consider the direct problem in inverse scattering where the governing physical equation is the Helmholtz equation. Inverse scattering has many scientific applications in medical imagining, non-destructive testing, as well as geophysics. We are particularly concerned with detecting small volume hidden objects within a complex media in the case where one can only make measurements on an exterior surface. Just as in the previous section, we assume that the domain $\Omega \subset \mathbb{R}^d$ (for $d=2,3$) is a bounded, simply connected open set with Lipschitz boundary $\partial \Omega$ with unit outward normal $\nu$. We let $D \subset \Omega$ with Lipschitz boundary $\partial D$ satisfying \eqref{sball}. Now, let the scattered field $u^s \in H^{1}(\Omega)$ satisfy  
\begin{equation}\label{inv-scat}
\Delta u^s + k^2 u^s = \rho \chi_D \quad  \text{in} \quad \Omega \quad \text{and} \quad u^s \big \rvert_{\partial \Omega} = f
\end{equation}
for any given $f \in H^{1/2}(\partial \Omega )$ where once again $\chi_{(\cdot)}$ denotes the indicator function. We let $k$ denote the wavenumber where we assume $k^2$ is not a Dirichlet eigenvalue of $- \Delta$ in $\Omega$. With this assumption on the wave number, we have that \eqref{inv-scat} is well-posed provided that the source $\rho \in L^{\infty}(D)$. 
By equation \eqref{inv-scat} we have that the Cauchy data is such that $(f , \partial_{\nu} u^s) \in H^{1/2}(\partial \Omega) \times H^{-1/2}(\partial \Omega)$. 

In this section, we will develop a direct sampling method for solving the inverse shape problem. This method has been employed for other imaging modalities such as DOT \cite{DSM-DOT} and Electrical Impedance Tomography \cite{DSM-EIT}. See also \cite{DSM-harris-nguyen2,DSM-ito1,DSM-nf} for applications with near field measurements. MUSIC-type algorithms has also been extensively used for similar shape reconstruction problems in \cite{MUSIC-EM2,EIT-granados1,MUSIC-parkEM, MUSIC-park}. However, our method only requires pair of Cauchy data to recover the support of the source and also avoids matrix operations. Lastly, our method is also highly tolerant to noise.

\subsection{Direct Sampling Method}

We denote $u_0^s \in H^{1}(\Omega)$ as the lifting which solves the Helmholtz equation such that 
\begin{equation}\label{helm-lift}
\Delta u_0^s + k^2 u_0^s = 0 \enspace \text{in} \enspace \Omega \quad \text{and} \quad u_0^s \big \rvert_{\partial \Omega} = f.
\end{equation}
Therefore, $u_0^s$ satisfies the background problem \eqref{inv-scat} (i.e. without the forcing term) with Dirichlet data $f \in H^{1/2}(\partial \Omega)$ and wavenumber $k$. By our assumption on the wave number we have that \eqref{helm-lift} is also well-posed. We proceed by defining the Dirichlet Green's function for the Helmholtz equation for the known domain $\Omega$ as $\mathbb{G}_k(\cdot , z) \in H^{1}_{loc}(\Omega \setminus \brac{z})$, which is the unique solution to the boundary value problem 
\[ \Delta \mathbb{G}_k(\cdot , z) + k^2 \mathbb{G}_k(\cdot , z) = - \delta (\cdot , z) \enspace \text{in} \enspace \Omega \quad \text{and} \quad \mathbb{G}_k (\cdot , z) \big \rvert_{\partial \Omega} = 0.
\] 
Here, we again assume that the wavenumber $k$ is as in \eqref{inv-scat} and \eqref{helm-lift}. For any fixed $z \in \Omega$, we appeal to Green's 2nd Theorem to obtain the representation 
\begin{align*}
	 -(u^s - u_0^s)(z) &= \int_{\Omega} (u^s - u_0^s)(x) \big [ \Delta \mathbb{G}_k(x , z) + k^2 \mathbb{G}_k(x , z) \big ] \, \text{d}x \\
	 &= \int_{D} \mathbb{G}_k(x,z) \rho (x) \, \text{d}x
\end{align*}
where we used the indicator function from our source term. By taking the normal derivative, we have that for all $z \in \partial \Omega$
\begin{equation}
\partial_{\nu} (u^s - u_0^s)(z) = - \int_{D} \rho (x) \partial_{\nu (z)}  \mathbb{G}_k (x , z) \, \text{d}x
\end{equation}
where the integrand is well defined since $z \in \partial \Omega$. Again, we let $\partial_{\nu (z)}$ denote the normal derivative on $\partial \Omega$ with respect to $z$. We now begin our asymptotic analysis of the normal derivative where $D$ is the finite union of small volume regions as given by \eqref{sball}. The following lemma is key in deriving the asymptotic expansion.
\begin{lemma}\label{inv-scat-exp}
For all $z \in \partial \Omega$ where $D$ is given by \eqref{sball}, we have that as $\epsilon \rightarrow 0$ 
$$ \int_{D} \rho (x) \mathbb{G}_k (x,z) \, dx = - \epsilon^d \sum_{j=1}^{J} |B_j| Avg(\rho_j) \partial_{\nu (z)} \mathbb{G}_k (x_j , z)  + \mathcal{O}(\epsilon^{d+1})$$ where $Avg(\rho_j)$ is the average value of $\rho$ in $D_j$.
\end{lemma}
\begin{proof}
By \eqref{sball}, we have that $x \in D_j$ if and only if $x = x_j + \epsilon y$ for some $y \in B_j$. Since $z \in \partial \Omega$, then $\partial_{\nu (z)} \mathbb{G}(\cdot , z)$ is smooth in the interior of $\Omega$ by elliptic regularity. Therefore, we have that for all $x \in D_j$ 
$$ \partial_{\nu (z)} \mathbb{G}_k (x , z) = \partial_{\nu (z)} \mathbb{G}_k (x_j + \epsilon y , z) = \partial_{\nu (z)} \mathbb{G}_k(x_j , z) + \mathcal{O}(\epsilon)$$
as $\epsilon \rightarrow 0$  by appealing to Taylor's Theorem. From this, we obtain that 
\begin{align*}
\int_{D} \rho (x) \partial_{\nu (z) } \mathbb{G}_k (x , z) \, \text{d}x &= \sum_{j=1}^{J} \int_{D_j} \rho (x)  \partial_{\nu (z) } \mathbb{G}_k (x_j + \epsilon y , z) \, \text{d}x \\
&= \sum_{j=1}^{J} \Big (\partial_{\nu (z) } \mathbb{G}_k (x_j, z) + \mathcal{O}(\epsilon)   \Big ) \int_{D_j} \rho (x) \, \text{d}x
\end{align*}
Therefore, we have that $$\int_{D} \rho (x) \partial_{\nu (z) } \mathbb{G}_k (x , z) \, \text{d}x = \epsilon^d \sum_{j=1}^{J} |B_j| \text{Avg}(\rho_j) \partial_{\nu (z)} \mathbb{G}_k (x_j , z) + \mathcal{O}(\epsilon^{d+1}) $$ as $\epsilon \rightarrow 0$ where we used the fact that $|D_j| = \epsilon^d |B_j|$ and $\text{Avg}(\rho_j)$ denotes the average value of $\rho$ in $D_j$.
\end{proof}  
From the above lemma, it is clear that for a specified $z \in \partial \Omega$, the normal derivative of the difference of $u^s$ and the lifting $u_0^s$ is approximated by the centers of the inclusions.

\begin{theorem}\label{inv-scat-main-thm}
For any $z \in \partial \Omega$ we have that 
$$ \partial_{\nu (z)} u^s (z) = \partial_{\nu (z)} u_0^s (z) - \epsilon^d \sum_{j=1}^{J} |B_j| \text{Avg}(\rho_j) \partial_{\nu (z)} \mathbb{G}_k (x_j , z) + \mathcal{O}(\epsilon^{d+1}) \quad \text{as} \quad \epsilon \rightarrow 0$$ provided that $D$ satisfies \eqref{sball}.
\end{theorem}
With this approximation to the Neumann data, we develop an algorithm that detects the centers of small volume regions within our domain. We now study a direct sampling method. This is done by using Theorem \ref{inv-scat-main-thm} and evaluating the reciprocity gap functional $R[v]$ given by \eqref{rgf}, where the Cauchy data $(u^s=f, \partial_{\nu}u^s)$ on $\partial \Omega$ is fixed. Recall, that we assume that $v \in H^1(\Omega)$ solves the Helmholtz equation in $\Omega$ which gives that 
\begin{align*}
R[v]  &=  \int_{\partial \Omega} v \partial_{\nu} u^s - u^s \partial_{\nu}v \, \text{d}s \\
	&= \int_{\partial \Omega} v \Big [ \partial_{\nu} u_0^s - \epsilon^d \sum_{j=1}^{J} \text{Avg}(\rho_j) |B_j| \partial_{\nu (z)} \mathbb{G}(x_j , z) + \mathcal{O}(\epsilon^{d+1})  \Big ] - u_0^s \partial_{\nu} v\, \text{d}s \\
	&=  - \epsilon^{d} \sum_{j=1}^{J} |B_j| \text{Avg}(\rho_j) \int_{\partial \Omega} v \partial_{\nu (z)} \mathbb{G}(x_j , z) \, \text{d}s + \mathcal{O}(\epsilon^{d+1}) \\
	&= \epsilon^d \sum_{j=1}^{J} |B_j| \text{Avg}(\rho_j) v(x_j) + \mathcal{O}(\epsilon^{d+1})
\end{align*} 
where we used \eqref{helm-lift} as well as the fact that $u^s \big \rvert_{\partial \Omega} = u_0^s \big \rvert_{\partial \Omega} = f$. 

Notice, we can take $v = \text{e}^{\text{i}kz \cdot \hat{y}}$, which is clearly a solution to the Helmholtz equation for all $z \in  \R^d$, when $\hat{y} \in \mathbb{S}^{d-1}$(i.e. unit circle/sphere). We proceed by defining the imaging functional $W(z): \R^d \rightarrow \mathbb{R}_{\geq 0}$ as
\begin{align}\label{dsm-func}
W(z) = \left | \big(R [\text{e}^{\text{i}kz \cdot \hat{y}}] , \text{e}^{\text{i}kz \cdot \hat{y}} \big)_{L^2 (\mathbb{S}^{d-1})} \right|.
\end{align}
This functional can be used to recover the region $D$ by plotting it's values in $\Omega$. To prove this fact, we will study the resolution analysis for this imaging functional. This will involve using the asymptotic expansion derived in Theorem \ref{inv-scat-main-thm} to write the functional in terms of Bessel functions. To this end, notice that 
\begin{align*}
  W(z) &= \left| \left( \epsilon^d \sum_{j=1}^{J} |B_j| \text{Avg} (\rho_j) \text{e}^{\text{i}kx_j \cdot \hat{y}}+\mathcal{O}(\epsilon^{d+1}), \text{e}^{\text{i}kz \cdot \hat{y}} \right)_{L^2 (\mathbb{S}^{d-1})} \right| \\
  &=  \left| \epsilon^d \sum_{j=1}^{J} \text{Avg} (\rho_j) |B_j| \int_{\mathbb{S}^{d-1}} \text{e}^{\text{i}k(x_j - z) \cdot \hat{y}} \, \text{d}s (\hat{y}) \right|+ \mathcal{O}(\epsilon^{d+1})\
\end{align*} 
where we used straightforward calculations and the asymptotic expansion of the reciprocity gap functional. Now, we will recall the Funk--Hecke integral identity 
\begin{equation*}
\int_{\mathbb{S}^{d-1}} \text{e}^{-\text{i}k(z-x)\cdot \hat{y}} \, \text{d}s(\hat{y})=
    \begin{cases}
        2\pi J_0(k|x-z|), & \text{in } \mathbb{R}^2,\\
        \\
        4\pi j_0(k|x-z|), & \text{in } \mathbb{R}^3
    \end{cases} 
\end{equation*}
see for e.g. \cite{multifreq-dsm,DSM-liu1}.
Therefore, it is now clear that for all $z \in \Omega$, 
\begin{equation}\label{w-inv-scat}
W(z) = \begin{cases} 
     \pm \epsilon^2 \, 2 \pi  \displaystyle \sum_{j=1}^{J} \text{Avg} (\rho_j) |B_j| J_0 (k |x_j - z| ) + \mathcal{O}(\epsilon^3), & d= 2 \\
     \\
    \pm  \epsilon^3 \, 4 \pi \displaystyle \sum_{j=1}^{J} \text{Avg} (\rho_j) |B_j| j_0 (k |x_j - z| ) + \mathcal{O}(\epsilon^4), &  d= 3
   \end{cases}
\end{equation}
where $J_0$ represents the zeroth order Bessel function of the first kind and $j_0$ represents the zeroth order spherical Bessel function of the first kind. This allows us to provide our main result of this section.
\begin{theorem}\label{inv-scat-thm}
Up to leading order, if $\text{Avg} (\rho_j)  \neq 0$ we have that for all $z \in \R^d \setminus D$, 
$$W(z) = \mathcal{O} \Big ( \text{dist}(z,\mathcal{X})^{\frac{1-d}{2}} \Big ) \quad \text{as} \quad \text{dist}(z,\mathcal{X}) \rightarrow \infty$$ 
provided that the region $D$ satisfies \eqref{sball}, where the set $\mathcal{X} = \brac{x_j : 1, \hdots, J}$.
\end{theorem}
\begin{proof}
In order to prove the result, we use the fact that 
$$J_0 (|z-x|)= \mathcal{O}\left(|z-x|^{-1/2}\right) \quad \text{ and } \quad j_0 (|z-x|)= \mathcal{O}\left(|z-x|^{-1}\right) $$ 
as $|z-x| \to \infty$ for the case when $d=2$ or $3$ along with the expansion in \eqref{w-inv-scat}.
\end{proof}

Thus, Theorem \ref{inv-scat-thm} can be used to recover the centers of the subregions since the imaging functional $W(z)$ attains a local maximum at each of the centers. We now introduce a lemma regarding the stability  for the reciprocity gap functional. This will help us obtain a stability estimate for $W(z)$.
\begin{lemma}
For added random noise $0 < \delta < 1$, we have that for any solution $v \in H^{1}(\Omega)$ to the Helmholtz equation,
$$\big| R[v] - R^{\delta}[v] \big| \leq C \delta \|{v}\|_{H^{1}(\Omega)}$$
where $R[v]$ is given by \eqref{rgf} and the perturbed reciprocity gap functional is given by
\begin{equation}\label{rgf-error}
R^{\delta}[v] = \int_{\partial \Omega} v \partial_{\nu} u_{\delta}^{s} - u_{\delta}^{s} \partial_{\nu} \, ds
\end{equation} 
provided that there are positive constants $C_1$ and $C_2$ such that
$$\| \partial_{\nu} ( u_{\delta}^{s} -  u^{s}) \|_{H^{-1/2}(\partial \Omega)}\leq C_1 \delta \quad \text{ and } \quad \| u_{\delta}^{s} - u^{s} \|_{H^{1/2}(\partial \Omega)}\leq C_2 \delta .$$
\end{lemma}
\begin{proof} 
By simply subtracting the expressions, we have that 
\begin{align*}
R[v] - R^{\delta}[v] &=  \int_{\partial \Omega} v ( \partial_{\nu} u^s - \partial_{\nu} u_{\delta}^s) - (u^s - u_{\delta}^s ) \partial_{\nu}v \, \text{d}s
\end{align*} 
Thus, we can estimate the above quantity such that
\begin{align*}
|R[v] - R^{\delta}[v]| &\leq \Big( \|{u_{\delta}^{s} - u^{s}}\|_{H^{1/2}(\partial \Omega)} \|{\partial_{\nu} v}\|_{H^{-1/2}(\partial \Omega)} \\
& \hspace{0.4in}+ \|{v}\|_{H^{1/2}(\partial \Omega)} \|{ \partial_{\nu} (u_{\delta}^{s} - u^{s})}\|_{H^{-1/2}(\partial \Omega)} \Big) \leq C \delta \| v\|_{H^1(\Omega)} 
 \end{align*} 
by the dual-pairing of $H^{1/2}(\partial \Omega)$ and $H^{-1/2}(\partial \Omega)$. We have also used Trace Theorems and the fact that $v$ solves Helmholtz equation in $\Omega$. This proves the claim. 
\end{proof}
We are now able to present the following theorem on the error estimate for the imaging functional $W(z)$.

\begin{theorem}\label{w-error}
For added random noise $0 < \delta < 1$, we have that for any $z \in \R^d$,
\begin{equation}
|W(z) - W^{\delta}(z)| = \mathcal{O}(\delta)\quad \text{ as } \quad \delta \to 0
\end{equation}
such that the perturbed imaging functional is defined as 
$$W^{\delta}(z) =  \Big |  \big( R^{\delta} [ e^{ikz \cdot \hat{y}}] , e^{ikz \cdot \hat{y}}  \big)_{L^2 (\mathbb{S}^{d-1})} \Big | $$ where $R^{\delta}[\cdot]$ is defined as in \eqref{rgf-error}.
\end{theorem}
\begin{proof}
By the Triangle and Cauchy-Schwarz inequalities, we have that 
$$|W(z) - W^{\delta}(z)| \leq \|{R [\text{e}^{\text{i}kz \cdot \hat{y}}] - R^{\delta} [ \text{e}^{\text{i}kz \cdot \hat{y}}]}\|_{L^{2}(\mathbb{S}^{d-1})} \|{\text{e}^{\text{i}kz \cdot \hat{y}}}\|_{L^{2}(\mathbb{S}^{d-1})}.$$ 
Note, that by the previous result in Lemma \ref{rgf-error}, we have that 
$$\|{R [\text{e}^{\text{i}kz \cdot \hat{y}}] - R^{\delta} [ \text{e}^{\text{i}kz \cdot \hat{y}}]}\|_{L^{2}(\partial \Omega)} \leq C \delta \| \text{e}^{\text{i}kz \cdot \hat{y}} \|_{H^1(\Omega)} . $$
Furthermore, we have that both $\| \text{e}^{\text{i}kz \cdot \hat{y}} \|_{H^1(\Omega)}$ and $ \|{\text{e}^{\text{i}kz \cdot \hat{y}}}\|_{L^{2}(\mathbb{S}^{d-1})}$ are bounded and independent of the parameter $\delta$. Thus, we have that
\begin{align*}
|W(z) - W^{\delta}(z)| \leq C \delta \quad \text{ as } \quad \delta \to 0
\end{align*}
which proves the claim. 
\end{proof}
This result demonstrates that the imaging functional $W(z)$ is stable with respect to error in the measured Cauchy data. This implies that plotting the imaging functional is an analytically rigorous as well as computationally simple and stable. 

%%%%%%%%%%%%%%%%%%%%%%%%%%%%%%%%%%%%%%%%%%%%%%% 
\subsection{Numerical Validation for the Direct Sampling Algorithm}
In this section, we provide some numerical examples for recovering the locations of the unknown source given by $\brac{x_j : j = 1, \cdots , J}$ using Theorem \ref{inv-scat-thm}. Just as in the previous section, all of our numerical experiments are once again done with the software \texttt{MATLAB} 2020a. For simplicity, {\it we let $\Omega$ be given by the unit ball in $\mathbb{R}^2$}. In order to do so, we first need a way to calculate the corresponding scattered field $u^s$ solving \eqref{inv-scat}. To this end, we can take the radiation scattered field for all $x \in \mathbb{R}^2$ given by 
$$x \longmapsto  - \int_{\mathbb{R}^2} \rho(y) \chi_{D}(y) \Phi_k (x,y) \, \text{d}y.$$ 
This scattered field solves the associated source problem in all of $\R^2$ where $\Phi_k$ denotes the radiating fundamental solution to the Helmholtz equation. 
Since $\chi_{(\cdot)}$ denotes the indicator function, we have that for all $x \in \Omega$ 
\begin{equation}\label{us-num}
u^s (x) = - \int_{D} \rho(y) \frac{\text{i}}{4} H_{0}^{(1)}(k |x-y|) \, \text{d}y
\end{equation} 
solves \eqref{inv-scat} with the corresponding Dirichlet data. It is a well known fact that the fundamental solution is given by 
$$ \Phi_k (x,y) = \frac{\text{i}}{4} H_{0}^{(1)}(k |x-y|)$$ 
where $H_{0}^{(1)}$ represents the first kind Hankel function of order zero. 

Next, we compute the normal derivative of the scattered field. 
%Using the fact that $\dfrac{d}{dt}H_{0}^{(1)}(t) = - H_{1}^{(1)} (t)$, 
It is straightforward to conclude that the normal derivative on $\partial \Omega$ of the solution $u^s$ is given by
\begin{equation}\label{n-us-num}
\partial_{\nu} u^s (x) = \int_{D} \rho(y) \frac{\text{i}k}{4} H_{1}^{(1)}(k |x-y|) \bigg [ \frac{1 - x \cdot y}{|x-y|}  \bigg ] \, \text{d}y
\end{equation}
where $H_{1}^{(1)}$ represents the first kind Hankel function of order one. We calculate the scattered field and its normal derivative as given by \eqref{us-num} and \eqref{n-us-num}, respectively, using the `\texttt{integral2}' command in \texttt{MATLAB}. Here, we evaluate the reciprocity gap functional $R[\text{e}^{\text{i}kx \cdot \hat{y}}]$ for 64 equally spaced points $\hat{y}$ on the unit circle. By appealing to our asymptotic result in \eqref{w-inv-scat}, the imaging functional is given by 
$$W_{\text{DIRECT}}(x) = \left| \big( R[ \text{e}^{\text{i}kx \cdot \hat{y}}] , \text{e}^{\text{i}kx \cdot \hat{y}} \big)_{L^2 (\mathbb{S}^{1})} \right|^p \quad \text{for any} \quad x \in \Omega$$ 
which is approximated via a Riemann sum using the `\texttt{dot}' command in \texttt{MATLAB}. In our calculations $p>0$ is a fixed chosen parameter to sharpen the resolution of the imaging functional. We also normalize the values of the imaging functional and pick $p=4$ in our calculations such that $W_\text{DIRECT}(x) = \mathcal{O}(1)$ for $x = x_j$ and $W_\text{DIRECT}(x) \approx 0$ for $x \neq x_j$.\\

In all our examples, we use the imaging functional $W_{\text{DIRECT}}(x)$ given above to recover the location of the components of the region $D$. In these experiments, the region $$ D = \bigcup_{j=1}^J \big( x_j + \epsilon B(0,1) \big) $$ with $B(0,1)$ being the unit circle centered at the origin. Here, we take $\epsilon = 0.01$ as well as adding random noise level $\delta$ to the simulate data $u^s$ and $\partial_\nu u^s$ on $\partial \Omega$. We let the wave number $k = 25$ and the points $x_j$ are points contained in the region $\Omega$. In Examples 1 and 2, $D$ is composed of two regions. In Example 3, $D$ is composed of three regions, and in Example 4, $D$ is composed of four regions. \\

\noindent{\bf Example 1: }\\
In our first example presented here, we let 
$$ x_1 = (0 , 0.75) \quad \text{and} \quad x_2 = (0.5 , 0)$$ for the reconstruction in Figure \ref{inv-direct1}. Presented is a contour and surface plot of the imaging functional $W_\text{DIRECT}(x)$. As we can see from the data tips, the imaging functional has spikes at the points $$ \widetilde{x}_1 = (0.0101, 0.7374) \quad \text{and} \quad \widetilde{x}_2 = (0.4949, 0.0101).$$ 
We can see that the locations of $\widetilde{x}_1 $ and $\widetilde{x}_2$ given by the Direct Sampling Algorithm provide an approximation for the locations of the components of the region $D$. Here we let noise level $\delta = 1\%$ and $\rho = 1$ in both subregions.\\
\begin{figure}[H]
\centering 
\includegraphics[scale=0.15]{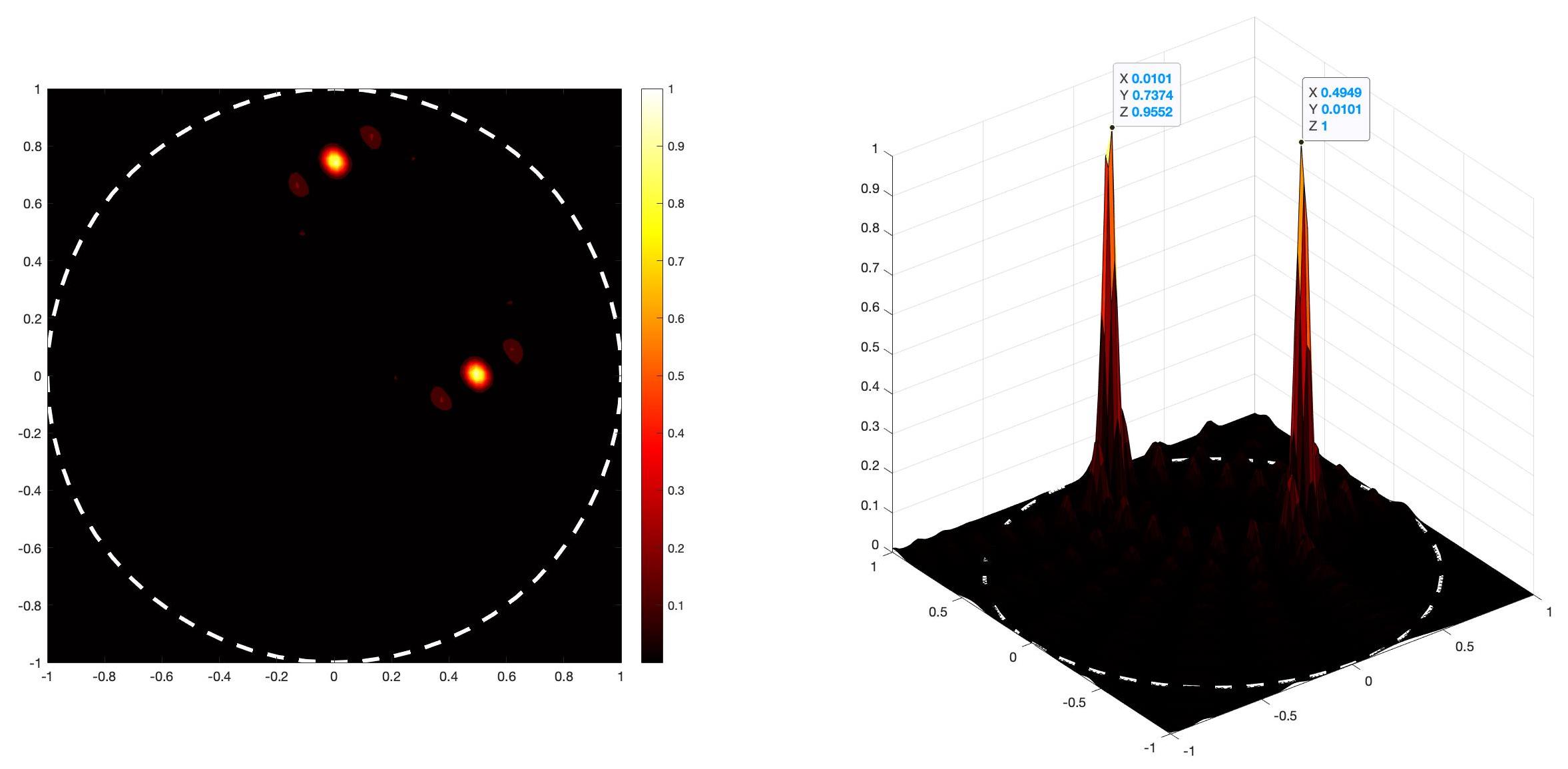}
\caption{Reconstruction of the locations $x_1 = (0 , 0.75)$ and $x_2 = (0.5 , 0)$ via the imaging function $W_{\text{DIRECT}}(x)$. Contour plot  on the left and Surface plot on the right.}
\label{inv-direct1}
\end{figure}

For the rest of the examples of this section, we vary the value of the forcing term $\rho$ on each of the components of $D$. Furthermore, we also increment the random noise level $\delta$ to demonstrate the stability of the method.\\

\noindent{\bf Example 2: }\\
For our second example presented here, we let 
$$ x_1 = (0.15 , 0.5) \quad \text{and} \quad x_2 = (0.35 , 0.2)$$ for the reconstruction in Figure \ref{inv-direct2}. Presented is a contour and surface plot of the imaging functional $W_\text{DIRECT}(x)$. As we can see from the data tips, the imaging functional has spikes at the points $$ \widetilde{x}_1 = (0.1717, 0.4747) \quad \text{and} \quad \widetilde{x}_2 = (0.3333, 0.2121).$$ 
Again, in this example we see that the locations of $\widetilde{x}_1 $ and $\widetilde{x}_2$ given by the direct sampling method provide an approximation for the locations of the components of the region $D$. Here we let noise level $\delta = 10\%$ where $\rho = 0.9$ in the region centered at $x_1$ and $\rho = 1$ in the region centered at $x_2$. In this example, notice that we have reduced the distance between $x_1$ and $x_2$ and incremented noise level $\delta$ from Example 1. Thus, the sharp reconstruction of $D$ as shown in Figure \ref{inv-direct2} illustrates the stability and robustness of this method.\\
\begin{figure}[H]
\centering 
\includegraphics[scale=0.15]{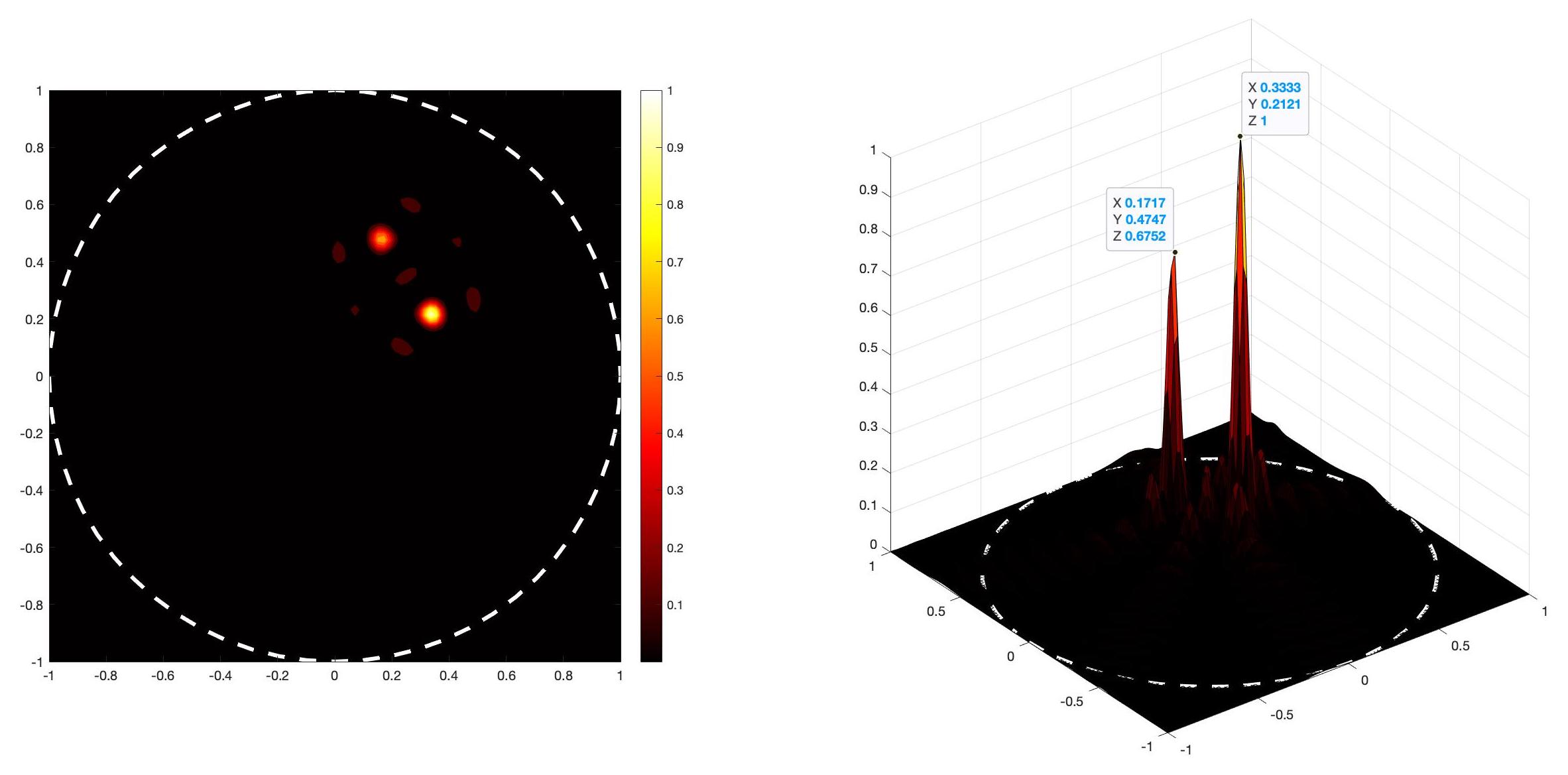}
\caption{Reconstruction of the locations $x_1 = (0.15 , 0.5)$ and $x_2 = (0.35 , 0.2)$ via the imaging functional $W_{\text{DIRECT}}(x)$. Contour plot  on the left and Surface plot on the right.}
\label{inv-direct2}
\end{figure}\

\noindent{\bf Example 3: }\\
In our third example presented here, we let
$$ x_1 = (-0.5 , -0.5), \quad x_2 = (0 , 0) \quad \text{and} \quad x_3 = (0.5, 0.25)$$ 
for the reconstruction in Figure \ref{inv-direct3}. Presented is a contour and surface plot of the imaging functional $W_\text{DIRECT}(x)$. As we can see from the data tips, the imaging functional has spikes at the points $$ \widetilde{x}_1 = (-0.5152, -0.5152), \quad \widetilde{x}_2 = (0.0101, 0.0101) \quad \text{and} \quad \widetilde{x}_3 = (0.4949, 0.2525).$$ In this example we see that the locations of $\widetilde{x}_1 $, $\widetilde{x}_2$ and $\widetilde{x}_3$ provide an approximation for the locations of the components of the region $D$. For this example we let $\delta = 20\%$ where $\rho = 0.8$ in the region centered at $x_1$, $\rho = 1.1$ in the region centered at $x_2$, and $\rho = 0.9$ in the region centered at $x_3$. \\
\begin{figure}[H]
\centering 
\includegraphics[scale=0.15]{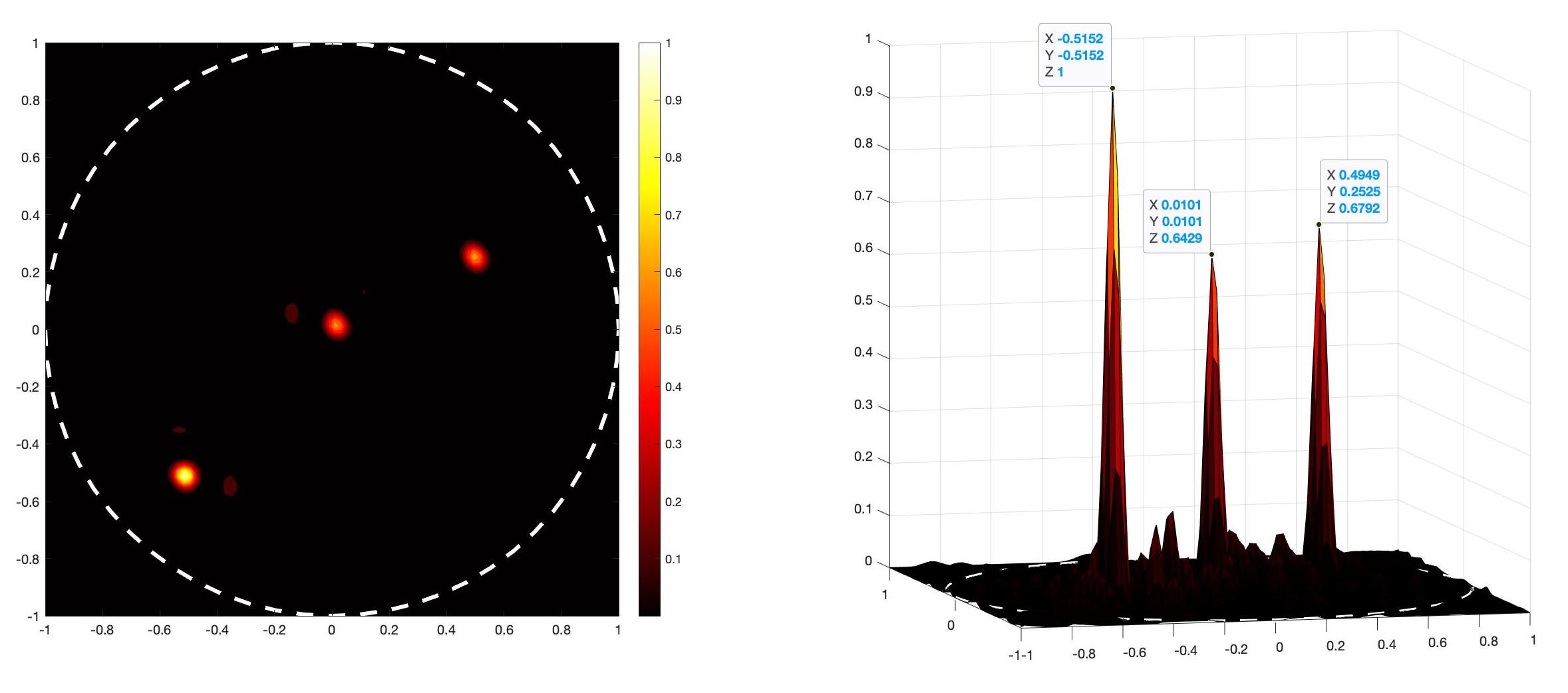}
\caption{Reconstruction of the locations $x_1 = (-0.5 , -0.5)$, $x_2 = (0 , 0)$ and $x_3 = (0.5, 0.25)$ via the imaging functional $W_{\text{DIRECT}}(x)$. Contour plot  on the left and Surface plot on the right.}
\label{inv-direct3}
\end{figure}

\noindent{\bf Example 4: }\\
In our final example presented here, we let
$$ x_1 = (0 , 0.5), \quad x_2 = (0.25 , 0.25), \quad x_3 = (-0.25 , -0.25) \quad \text{and} \quad x_4 = (0 , -0.75)$$ for the reconstruction in Figure \ref{dot-music4}. Presented is a contour and surface plot of the imaging functional $W_\text{DIRECT}(x)$. As we can see from the data tips, the imaging functional has spikes at the points 
$$ \widetilde{x}_1 = (0.0101, 0.4949), \quad \widetilde{x}_2 = (0.2525, 0.2727), \quad \widetilde{x}_3 = (-0.2525, -0.2727)$$
and 
$$\widetilde{x}_4 = (-0.0101, -0.7576).$$
In this example we see that the reconstructed locations provide an approximation for the locations of the region $D_j$. For this example we let $\delta = 25\%$ where $\rho = 0.95$ in the region centered at $x_1$, $\rho = 1$ in the region centered at $x_2$, $\rho = 0.9$ in the region centered at $x_3$, and $\rho = 1.1$ in the region centered at $x_4$.\\
\begin{figure}[H]
\centering 
\includegraphics[scale=0.15]{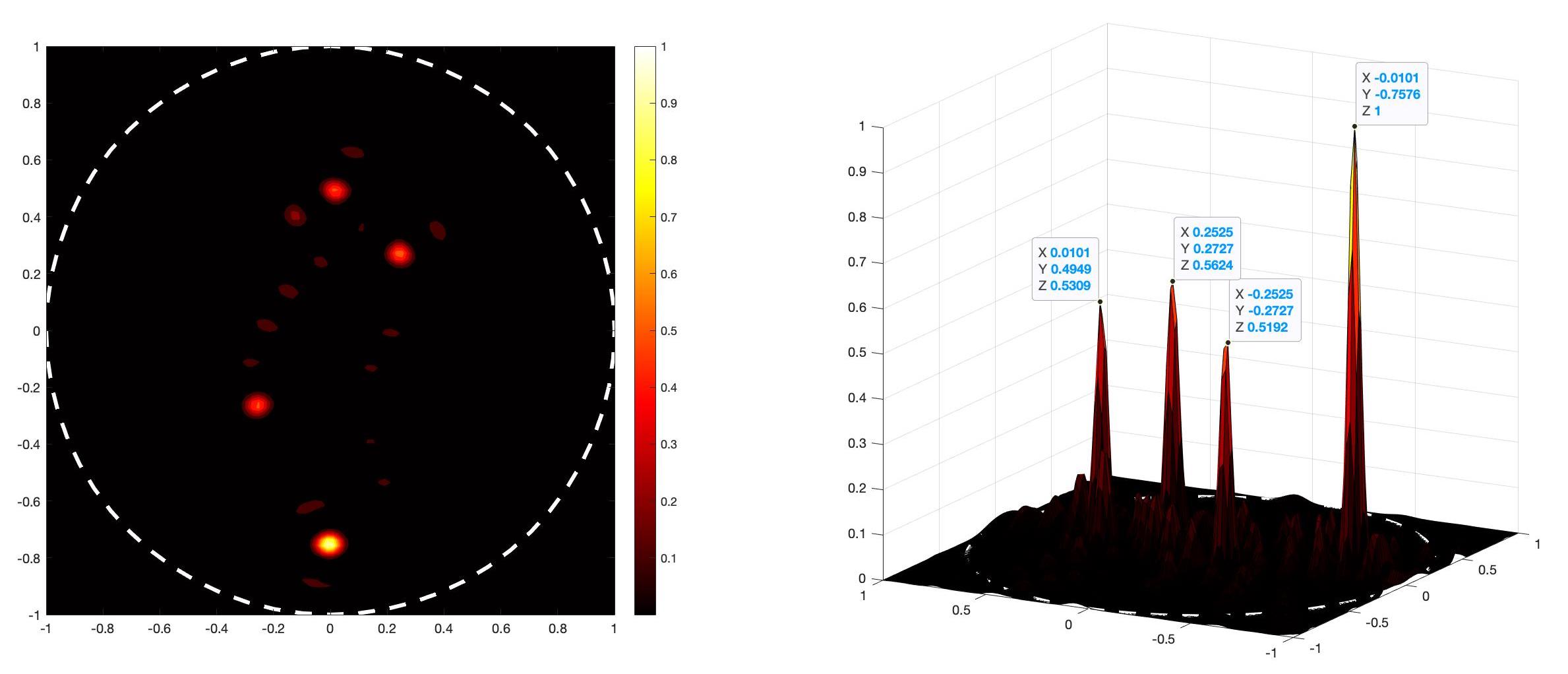}
\caption{Reconstruction of the locations $x_1 = (0 , 0.5)$, $x_2 = (0.25 , 0.25)$, $x_3 = (-0.25 , -0.25)$ and $x_4 = (0 , -0.75)$ via the imaging functional $W_{\text{DIRECT}}(x)$. Contour plot  on the left and Surface plot on the right.}
\label{inv-direct4}
\end{figure}

%%%%%%%%%%%%%%%%%%%%%%%%%%%%%%%%%%%%%%%%%%%%%%%%%%%%%%%%%%%
\section{Conclusions}\label{Conclusion}
In this paper, we studied the use of qualitative methods for small volume inverse shape problems in DOT and inverse scattering. In both cases, we analyzed the asymptotic expansion of the reciprocity gap functional \eqref{rgf} in order to construct an imaging functional to recover the region of interest $D$. For the DOT problem, we have studied the MUSIC algorithm. Whereas in the inverse scattering problem, we derived a direct sampling method. We note that the analysis provided here can be used to study the inverse scattering problems in $\R^d$ for $d=2,3$, where one can use \eqref{us-num} and the asymptotic analysis presented here. Both algorithms allow for fast and accurate reconstruction with little a priori knowledge of $D$. A future direction for this project, in the area of inverse scattering can be to study the problem in Section \ref{InvScat-section} for the case of electromagnetic and elastic scattering. Another interesting project would be to develop a direct sampling method as in \cite{DSM-DOT} for the DOT problem presented in Section \ref{DOT-section}. \\

\noindent{\bf Acknowledgments:} The research of G. Granados and I. Harris is partially supported by the NSF DMS Grant 2107891.

%%%%%%%%%%%%%%%%%%%%%%%%%%%%%%%%%%%%%%%%%%%%%%%%%%%%%%%%%%%

\end{document}